\documentclass[12pt]{amsart}
\usepackage{amsmath,amssymb,amsthm}
\usepackage{mathtools}

\usepackage{scalerel}
\usepackage{enumerate}
\usepackage{hyperref}
\hypersetup{colorlinks=true,citecolor=blue,linkcolor=red}
\usepackage{cite}
\usepackage{color}
\usepackage[usestackEOL]{stackengine}
\usepackage[a4paper,width=16.5cm,top=3cm,bottom=2.8cm]{geometry}
\usepackage{esint}

\newtheorem{theorem}{Theorem}[section]
\newtheorem{proposition}[theorem]{Proposition}
\newtheorem{corollary}[theorem]{Corollary}
\newtheorem{lemma}[theorem]{Lemma}
\newtheorem{definition}[theorem]{Definition}
\newtheorem{remark}[theorem]{Remark}

\theoremstyle{plain}
\newtheorem*{theorem*}{Theorem}
\newtheorem*{question*}{Question}
\newtheorem*{example*}{Example}

\def\dashint{\,\ThisStyle{\ensurestackMath{%
			\stackinset{c}{.2\LMpt}{c}{.5\LMpt}{\SavedStyle-}{\SavedStyle\phantom{\int}}}%
		\setbox0=\hbox{$\SavedStyle\int\,$}\kern-\wd0}\int}
\newcommand{\rn}{\mathbb{R}^n}
\newcommand{\rnm}{\mathbb{R}^{n+m}}
\newcommand{\rd}{\mathbb{R}^d}

\newcommand{\lprn}{{L}^p(\rn)}

\newcommand{\lprnmmu}{{L}^p(\rnm, \mu)}

\newcommand{\lpn}[1]{\|#1\|_{{L}^p}}

\newcommand{\lprnmmun}[1]{\|#1\|_{\lprnmmu}}

\newtheorem{example}{Example}
\newtheorem{claim}[theorem]{Claim}
\newcommand{\ra}{\rightarrow}
\newcommand{\R}{{\mathbb R}}
\newcommand{\Rnp}{{\R^{n+1}_+}}
\newcommand{\Rn}{{\R^n}}
\newcommand{\Rm}{{\R^m}}

\newcommand{\JNp}{{{\rm JN}_p}}
\newcommand{\JNq}{{{\rm JN}_q}}
\newcommand{\GRp}{{{\rm GaRo}_p}}
\newcommand{\GRinfty}{{{\rm GaRo}_\infty}}

\newcommand{\N}{{\mathbb N}}
\newcommand{\Q}{{\mathbb Q}}
\newcommand{\Z}{{\mathbb Z}}
\newcommand{\Zn}{{\Z^n}}
\newcommand{\loc}{{\rm loc}}

\newcommand{\Lone}{{L^1}}
\newcommand{\Lp}{{L^p}}
\newcommand{\Loneloc}{{L^1_\loc}}
\newcommand{\BMO}{{\rm BMO}}
\newcommand{\cD}{{\mathcal D}}
\newcommand{\cC}{{\mathcal C}}
\newcommand{\cP}{{\mathcal P}}
\newcommand{\cO}{{\mathcal O}}
\newcommand{\Opgamma}{{\cO^{p}_\gamma}}
\newcommand{\Op}{{\cO^{p}}}
\newcommand{\On}{{\cO^{n}}}
\newcommand{\Cmax}{{\cC_{\text{max}}}}
\newcommand{\Cmin}{{\cC_{\text{min}}}}
\newcommand{\cB}{{\mathcal B}}

\newcommand{\lmin}{{\ell_{\text{min}}}}
\newcommand{\lmax}{{\ell_{\text{max}}}}
\newcommand{\lab}{{\ell_{\alpha, \beta}}}
\newcommand{\labp}{{\ell_{\alpha', \beta'}}}
\newcommand{\lde}{{\ell_{\delta, \varepsilon}}}
\newcommand{\ldep}{{\ell_{\delta', \varepsilon'}}}

\newcommand{\ftil}{{\tilde{f}}}

\newcommand{\Ikm}{{I_{k,m}}}
\newcommand{\Ijm}{{I_{j,m}}}

\newcommand{\Qbar}{\bar{Q}}
\newcommand{\Ibar}{\bar{I}}
\newcommand{\cDtil}{\tilde{\cD}}

\title{A dyadic approach to weak characterizations of function spaces}
\author[Dafni]{Galia Dafni}
\address{(G.D.) Concordia University, Department of Mathematics and Statistics, Montr\'{e}al, Qu\'{e}bec, H3G 1M8, Canada}
\curraddr{}
\email{galia.dafni@concordia.ca}
\thanks{The authors were partially supported by the Natural Sciences and Engineering Research Council
(NSERC) of Canada, and the Centre de recherches math\'{e}matiques (CRM)}
\author[Shaabani]{Shahabdoddin Shaabani}
\address{(S.S.) Concordia University, Department of Mathematics and Statistics, Montr\'{e}al, Qu\'{e}bec, H3G 1M8, Canada}
\curraddr{}
\email{shahaboddin.shaabani@concordia.ca}

\begin{document}

\begin{abstract}
Weak-type quasi-norms are defined using the mean oscillation or the mean of a function on dyadic cubes, providing discrete analogues and variants of the corresponding quasi-norms on the upper half-space previously considered in the literature. Comparing the resulting function spaces to known function spaces such as $\dot{W}^{1,p}(\rn)$, $\JNp$, $\Lp$ and weak-$\Lp$ gives new embeddings and characterizations of these spaces.  Examples are provided to prove the sharpness of the results.
\end{abstract}

\maketitle

\section{Introduction}
There has been much interest in recent years in characterizing Sobolev spaces without using derivatives, and via weak-type norms.  We refer the reader to the exposition of these developments in the papers of Dom\'inguez and Milman \cite{Oscar} and Frank \cite{Frank}, and the references therein.  The work described here was motivated by the results in these papers, with the aim of applying dyadic analysis techniques to prove analogues in the discrete setting.  Such techniques not only simplify the proofs but also lead to new embeddings for known function spaces.

The function spaces concerned involve either the means or the mean oscillations of a locally integrable function $f$.  For $f\in L^1_{loc}(\rn,\mu)$, and any set $E$ of positive and finite measure $\mu(E)$, the mean of $f$ and the mean oscillation of $f$ on $E$ are given by
$$f_E:= \fint_E f d\mu := \frac{1}{\mu(E)} \int_E f\quad \mbox{and} \quad O(f,E) := \fint_E |f - f_E| d\mu.$$
Restricting $\mu$ to Lebesgue measure and $1\le p<\infty$,  we consider, as in \cite{Frank}, the function $m(f)$ and the measure $\nu_p$ on the upper half space $\mathbb{R}^{n+1}_+$ defined by
\[
m(f)(x,r):=O(f,B(x,r)),\quad d\nu_p(x,r):=\frac{dxdr}{r^{p+1}},\quad x\in \rn, \quad r>0.
\]
In \cite[Theorem 1]{Frank} it was shown that for $1<p<\infty$, $f$ is in the homogeneous Sobolev space  $\dot{W}^{1,p}(\rn)$ if and only if $m(f)$ is in the  weak-$\Lp$ space $L^{p,\infty}(\Rnp, \nu_p)$, and
\begin{equation}\label{frankresult}
	\|\nabla f\|_{\lprn}\approx \|m(f)\|_{L^{p,\infty}(\Rnp, \nu_p)}.
\end{equation}
Moreover,
$$c_{n,p}\|\nabla f\|^p_{\lprn}=\lim_{\lambda \ra 0} \lambda^p \nu_p(\{(x,t)\in \Rnp: |m(f)(x,t)| > \lambda\}).$$
In \eqref{frankresult} and below, for a measurable function $F$ on a measure space $(X,\nu)$, we use the notation
$$\|F\|_{L^{p,\infty}(X, \nu)}:= \left(\sup_{\lambda > 0} \lambda^p \nu(\{x\in X: |F(x)| > \lambda\})\right)^{\frac 1 p}.$$
We also use the notation $A \approx B$ for quantities $A, B$ with $A \lesssim B$ and $B \lesssim A$, where $A \lesssim B$ means there is a constant $c$ such that $A \le cB$.

In this paper we approach the notions in \cite{Frank} from a discrete point of view.  Such discrete analogues were previously considered in work of Rochberg and Semmes \cite{RS}, and Semmes together with Connes, Sullivan and Teleman \cite[Appendix]{Connes}, in connection with the weak Schatten-class condition on commutators $[f,T]$.  For Calder\'on-Zygmund singular integral operators $T$, the boundedness of such commutators on $L^p$, $1 < p < \infty$, is equivalent to $f$ having bounded mean oscillation (BMO) as defined in \cite{JN}, namely $\|f\|_{\BMO}: = \sup_{\mbox{\tiny cubes } Q} O(f, Q) < \infty$, 
and the compactness is equivalent to $f$ having vanishing mean oscillation (VMO).  A description of the problem and its history can be found in \cite[Section 1.1]{Frank}; see also recent results by Frank, Sukochev and Zanin \cite{FSZ}.

Let $\cD$ be the collection of all dyadic cubes in $\rn$. For $f\in L^1_{loc}(\rn)$, taking the sequence $\left\{O(f,Q)\right\}_{Q\in \cD}$ given by the mean oscillation of $f$ on dyadic cubes,
one can define the weak-type quasi-norm for $\gamma_1,\gamma_2 \in \R$ and $1 \le p < \infty$ by
\begin{equation}
\label{Opgamma}
\|f\|_{\cO^p_{\gamma_1,\gamma_2}(\rn)}:= \left(\sup_{\lambda>0}\lambda^p\sum_{\ell(Q)^\frac{\gamma_1}{p}|O(f,Q)|>\lambda}\ell(Q)^{-\gamma_2}|Q|\right)^{\frac{1}{p}},
\end{equation}
where $\ell(Q)$ denotes the sidelength of $Q$.  

The case $\gamma_1 = 0$, $\gamma_2 = p$, which we denote by $\|f\|_{\Op}$, can be thought of as the dyadic version of the weak norm on the right-hand side of \eqref{frankresult}, by considering a tiling of the upper half-space by the upper Carleson boxes $\{S(Q)\}_{Q\in \cD}$, $S(Q) = Q \times (\ell(Q), 2\ell(Q)]$, and noting that $\nu_p(S(Q)) \approx \ell(Q)^{-p}|Q|$.  
The inequality
\begin{equation}
\label{Opmf}
\|f\|_{\Op(\rn)}\lesssim \|m(f)\|_{L^{p,\infty}(\Rnp, \nu_p)},\quad 1\le p<\infty,
\end{equation}
holds (see Proposition~\ref{prop-contdyadic}) but the two quasi-norms are not equivalent, as can be seen by considering the function $f=\chi_{\mathbb{R}_+\times \mathbb{R}^{n-1}}$ which has 
$\|m(f)\|_{L^{p,\infty}(\Rnp, \nu_p)} > 0 = \|f\|_{\Op(\rn)}$.  

As shown in Proposition~\ref{prop-contdyadic}, the reverse inequality in \eqref{Opmf} holds if one replaces the classical dyadic grid by $n+1$ dyadic systems, allowing one to recover the embeddings of $\Op(\rn)$ from the classical Sobolev embeddings via \eqref{frankresult}.  In Section~\ref{Section1} we show that it is possible to obtain certain weak Sobolev-Poincar\'e and Moser-Trudinger inequalities for functions in $\Op$ by purely dyadic arguments not involving derivatives.

Inspired by a question of Oscar Dom\'inguez, in Section~\ref{Section2} we consider variants of these spaces corresponding to the case $\gamma_1 = \gamma_2 = \gamma$ in \eqref{Opgamma}, which we denote by $\Opgamma$, and show (see Theorem~\ref{embeddings}) the embedding
\begin{equation}
\label{JNOp}
\JNp(\Rn, \cD) \subset \Opgamma(\Rn), \gamma \not\in [0,n], 1 < p<\infty,
\end{equation}
as well as the corresponding embedding of $\JNp(\Rn)$ into the analogue of $\Opgamma$ defined by a weak-type norm in the upper half-space (see Corollary~\ref{contembeddings}).  The notation here refers to the John-Nirenberg space $\JNp(\rn)$, $1<p<\infty$, which consists of all locally integrable functions $f$ such that
\[
\|f\|_{\JNp(\rn)}:=\left(\sup_{\cP}\sum_{Q\in\cP}|Q|O(f,Q)^p\right)^\frac{1}{p}<\infty,
\]	
where the supremum is taken over all disjoint collections of cubes $\cP$.  If the collections are restricted to dyadic cubes, we denote the resulting space by $\JNp(\rn,\cD)$.  Analogously, for a cube $Q_0$, we can define $\JNp(Q_0)$ (resp.\ $\JNp(Q_0,\cD(Q_0))$) by restricting to disjoint collections of subcubes (resp.\ dyadic subcubes) of $Q_0$.  The definition of this space is motivated by a result of John and Nirenberg, \cite[Lemma 3]{JN}, which gives the embedding
\begin{equation}
\label{JN}
\JNp(Q_0,\cD) \subset L^{p,\infty}(Q_0).
\end{equation}

The embedding \eqref{JNOp} is new and completely independent of \eqref{JN}, as there are no containment relations between $\Opgamma$ and weak-$L^p$, and in fact $\Opgamma$ is not contained in weak-$L^q$ for any $q > 1$, nor is it rearrangement invariant.  Moreover, the embedding in  \eqref{JNOp} is proper and is not valid for $\gamma \in [0,n]$.  These claims are shown by means of examples in Section~\ref{Examples}.  Since $L^p$ is  contained in $\JNp$,  and this inclusion is proper, as shown in \cite{DHKY}, \eqref{JNOp} implies, and is strictly stronger than, the embedding $L^p \subset \Opgamma$.

A characterization of $\Lp(\Rn)$, $1 \le p < \infty$, namely two-sided embeddings, is also provided in this paper by means of a variant of the definition of the weak-type quasi-norms in \eqref{Opgamma}, where the mean oscillation $O(f,Q)$ is replaced by the mean $f_Q$.  This characterization is valid for certain values of the parameter $\gamma$ and a measure $\mu$ which can be more general than Lebesgue measure (see Theorem~\ref{embeddings} and Theorem~\ref{cubelptheorem}) and gives a range of discrete analogues of \cite[Theorem 1.1]{Oscar} in the case of averages over balls, without using the maximal function.  In the case of nonnegative functions, this lead (see Corollary~\ref{contembeddings}) to the continuous version of this identification, namely the characterization of $L^p$ via a weak-type norm in the upper half-space.  In Section~\ref{Biparameter}, a bi-parameter version of the discrete space is considered and once again identified with $L^p$ (see Theorem~\ref{thm-biparameter}).

The authors wish to thank Oscar Dom\'inguez for posing the question which motivated some of the research in the paper, and for interesting mathematical discussions. The authors are also grateful for the referee's comments.

\section{Weak Sobolev-Poincar\'e and Moser-Trudinger inequalities for $\Op$ }
\label{Section1}

Let $Q_0$ be a cube, $f \in \Lone(Q_0)$, and $\Op(Q_0)$ be the space obtained by taking $\gamma_1 = 0$, $\gamma_2 = p$ in \eqref{Opgamma} and replacing the collection of dyadic cubes $\cD(\Rn)$ by $\cD(Q_0)$, the collection of cubes obtained by repeatedly bisecting the sides of $Q_0$, which we will call the dyadic subcubes of $Q_0$.  As explained in the introduction, we cannot assume that functions in $\Op(Q_0)$ are differentiable.
Nevertheless, we can prove that when $1\le p<n$, the following weak Sobolev-Poincar\'e inequality holds.  
\begin{theorem}
\label{SPthm}
	Let $1\le p<n$, $p^*=\frac{np}{n-p}$, and $|Q_0|=1$. Then we have
	\begin{equation}\label{Sobolevembeding}
		\|f-f_{Q_0}\|_{L^{p^*,\infty}(Q_0)}\lesssim \|f\|_{\Op(Q_0)}.
	\end{equation}
\end{theorem}

In the proof of the theorem we will need to use the Garsia-Rodemich characterization of weak $\Lp$.  The Garsia-Rodemich space $\GRp(Q_0)$, $1 \le p < \infty$, consists of $f \in L^1(Q_0)$ such that 
	\begin{equation}\label{GR}
	\sum_{Q\in\cP}\frac{1}{|Q|} \int_Q\int_Q |f(x) - f(y)| dx dy \le C\left(\sum_{Q\in\cP}|Q|\right)^{1/p'}, \quad \frac 1 p + \frac 1 {p'} = 1,
	\end{equation}
for all disjoint collections $\cP$ of subcubes of $Q_0$.  We denote the smallest $C$ for which this inequality holds by $\|f\|_{\GRp(Q_0)}$, and analogously define the dyadic version $\GRp(Q_0, \cD)$ by restricting to $\cP \subset \cD(Q_0)$.  Since
$$\fint_Q\fint_Q |f(x) - f(y)| dx dy \approx O(f,Q),$$
the left-hand side of \eqref{GR} can be replaced by $\sum_{Q\in\cP}|Q| O(f,Q)$.  As shown in \cite{Milman1}, an application of H\"older's inequality gives $\|f\|_{\GRp(Q_0)} \lesssim \|f\|_{\JNp(Q_0)}$.  The same inequality holds in the dyadic case.

Generalizing the one-dimensional case proved in \cite{GaRo} to $n$ dimensions, the main result of \cite{Milman1} is that 
$$L^{p,\infty}(Q_0) = \GRp(Q_0), \quad 1 < p < \infty.$$
The proof is not given in the dyadic setting, but one can replace $\GRp(Q_0)$ by $\GRp(Q_0, \cD)$.  To see this it suffices to show 
	\begin{equation}
	\label{GRdyadic}
	\|f-f_{Q_0}\|_{L^{p,\infty}(Q_0)}\lesssim \|f\|_{\GRp(Q_0, \cD)}.
		\end{equation}
In the proof given in \cite[Section 4]{ABKY} of the weak-type inequality for the embedding \eqref{JN},
$$	\|f-f_{Q_0}\|_{L^{p,\infty}(Q_0)}\lesssim \|f\|_{\JNp(Q_0)},
$$
 the key is a good-$\lambda$ inequality, Lemma 4.5, which is formulated using $K_f = \|f\|_{\JNp(Q_0)}$.  As the proof of this result relies on the Calder\'on-Zygmund decomposition, which uses only cubes in $\cD(Q_0)$, this inequality is valid  with the smaller value $K_f = \|f\|_{\JNp(Q_0,\cD)}$ . In fact, it is also valid with the even smaller $K_f =  \|f\|_{\GRp(Q_0, \cD)}$.  To see this it is only a matter of avoiding the use of H\"older's inequality at the end of the proof to obtain
$$|E_{Q_0}(\lambda)| \le  \frac{\sum |Q_k|O(f,Q_k)}{(1 - 2^nb)\lambda}  
 \le
\frac{|E_{Q_0}(b\lambda)|^{1/p'}}{(1 - 2^nb)\lambda}\|f\|_{\GRp(Q_0, \cD)}.$$

We are now ready to prove Theorem~\ref{SPthm}.
\begin{proof}[Proof of  Theorem~\ref{SPthm}]
	Let us normalize $f$ and assume that $\|f\|_{\Op(Q_0)}=1$. Then from the Garsia-Rodemich characterization of $L^{p^*,\infty}(Q_0)$, and in particular \eqref{GRdyadic}, it suffices to show that for any disjoint collection $\cP \subset \cD(Q_0)$ we have
	\[
	\sum_{Q\in \cP}|Q|O(f,Q)\lesssim |\cup\cP|^{1-\frac{1}{p^*}}.
	\]
	So, let us partition $\cP$ as 
	\[
	\cP_k=\left\{Q\in\cP| 2^k<O(f,Q)\le 2^{k+1}\right\}, \quad k\in \mathbb{Z},
	\]
	and note that we have
	\begin{equation}\label{Ck}
		\sum_{Q\in \cP}|Q|O(f,Q)\le \sum_{k\in \mathbb{Z}}2^{k+1}\sum_{Q\in {\cP}_k}\ell(Q)^n.
	\end{equation}
Now using the assumption that $\|f\|_{\Op(Q_0)}=1$ gives us
\[
\sum_{Q\in {\cP}_k}\ell(Q)^n\le\left(\sum_{Q\in {\cP}_k}\ell(Q)^{n-p}\right)^{\frac{n}{n-p}}\le \left(\sum_{O(f,Q) > 2^k}\ell(Q)^{-p}|Q|\right)^{p^*} \le 2^{-kp^*}.
\]
Breaking up the summation over $k$ in \eqref{Ck} at an arbitrary $\eta>0$, we obtain
\[
	\sum_{Q\in \cP}|Q|O(f,Q)\le2 \sum_{2^k<\eta}2^k|\cup\cP|+2\sum_{2^k\ge\eta}2^{k(1-p^*)}\lesssim \eta|\cup\cP|+\eta^{1-p^*}.
	\]
This is optimized by the choice $\eta=|\cup\cP|^{-\frac{1}{p^*}}$, which gives
	\[
	\sum_{Q\in \cP}|Q|O(f,Q)\lesssim|\cup\cP|^{1-\frac{1}{p^*}}
	\]
and finishes the proof.
\end{proof}

If in the proof of Theorem~\ref{SPthm} we had replaced the left-hand side of \eqref{Ck} by $\sum_{Q\in \cP}|Q|O(f,Q)^q$, for $q < p^*$, we would immediately get  the embedding of $\Op(Q_0)$ into $\JNq(Q_0)$.  This is not as strong as the embedding into $L^q(Q_0)$ implied by the result of the theorem, since the inclusion $L^q(Q_0) \subset \JNq(Q_0)$ is proper, as shown in \cite{DHKY}.

The case $p = \infty$ in the definition of the Garsia-Rodemich space, \eqref{GR}, makes sense, and it is shown in \cite[Lemma 2.3]{Milman2} that
$$\GRinfty(Q_0) = \BMO(Q_0).$$
Corresponding to this in the Sobolev embedding is the case $p = n$.  In the dyadic setting, the assumption that $f \in \On(Q_0)$ says that
\begin{equation}
\label{eqnlambda}
\#\{Q \in \cD(Q_0): O(f,Q) > \lambda\} \le \lambda^{-n} \|f\|^n_{\On(Q_0)}, \quad \lambda > 0.
\end{equation}
Applying this to $\lambda > \|f\|_{\On(Q_0)}$ means the number of cubes on the left-hand side is zero 
and therefore $f$ is in dyadic BMO on $Q_0$ with
$$\|f\|_{\BMO(Q_0,\cD(Q_0))}: =\sup_{Q \in \cD(Q_0)} O(f, Q) \le  \|f\|_{\On(Q_0)}.$$

In the following theorem we give a more refined version of this embedding for the case $p = n$, namely a Moser-Trudinger-type inequality.
\begin{theorem}
\label{MTthm}
Let $Q_0$ be a cube with $|Q_0|=1$. Then for $p=n$ we have
	\begin{equation}\label{trudinger}
	\|f-f_{Q_0}\|_{e^{L^{\frac{n}{n-1}}}}\lesssim \|f\|_{\On(Q_0)},
	\end{equation}
and for $n<p<\infty$ we have
\begin{equation}\label{moryineq}
	\|f-f_{Q_0}\|_{L^{\infty}(Q_0)}\lesssim \|f\|_{\Op(Q_0)}.
\end{equation}
\end{theorem}

For the proof of this theorem we will need to recall a mean oscillation inequality due to Lerner.  We give here a simpler version, found in \cite{Lerner}, of the local mean oscillation inequality.   A more general, global version is given in \cite[Theorem 10.2]{LernerNazarov}.

\begin{lemma}[Lerner]
\label{lerner}
	For $f\in L^1(Q_0)$, there exists a sequence $\{\cP_k\}$ of collections of disjoint dyadic cubes, which is contracting  in the sense that each cube in $\cP_k$ is contained in a cube in $\cP_{k-1}$, and $|\cup \cP_k|\le 2^{-k}|Q_0|$, for which the following pointwise domination holds:
	
	\[
	|f-f_{Q_0}|\lesssim \sum_{k=1}^{\infty}\sum_{Q\in\cP_k} O(f,Q)\chi_Q.
	\]
\end{lemma}

\begin{proof}[Proof of Theorem~\ref{MTthm}]
	To prove \eqref{trudinger}, we may assume $\|f\|_{\On(Q_0)}=1$.  As pointed out before the statement of the theorem, this implies that $O(f, Q) \le 1$ for all $Q \in \cD(Q_0)$. 

By Lemma~\ref{lerner}, for $\lambda > 0$ and any positive integer $K$,
$$\Big\{|f-f_{Q_0}| > \lambda\Big\} \subset \Big\{\sum_{k=1}^{K}\sum_{Q\in\cP_k} O(f,Q)\chi_Q > \lambda\Big\} + \Big\{\sum_{k=K+1}^{\infty}\sum_{Q\in\cP_k} \chi_Q > 0\Big\}$$
and the measure of the second term on the right is dominated by $\sum_{k=K+1}^{\infty}|\cup \cP_k| \le 2^{-K}$.

If $x$ belongs to the first set  on the right then there is a finite sequence  of dyadic cubes $Q_1,Q_2, \ldots,Q_K$ containing $x$ with $Q_k \in \cP_k$ 
and, by \eqref{eqnlambda}, we have 
	\[
	 \#\left\{k|O(f,Q_k)>2^{-l}\right\}\le \min\left\{2^{ln},K\right\},\quad l\in \N,
	\]
	which gives us
	\[
	\sum_{k=1}^{K}O(f,Q_k)\le \sum_{2^{ln} \le K}2^{ln}2^{-l+1}   +  \sum_{2^{ln} > K} K2^{-l+1} \le CK^{\frac{n-1}{n}}.
	\]
Letting $\lambda = CK^{\frac{n-1}{n}}$, we see that the first set must be empty and therefore
		\[
	\Big|\left\{|f-f_{Q_0}|\ge CK^{\frac{n-1}{n}}\right\}\Big| \le 2^{-K}.
	\]
	Since the estimate holds for all positive integers $K$,  this completes the proof of \eqref{trudinger}. 
	
	To prove the second inequality, for $p > n$, one first applies the analogue of \eqref{eqnlambda} for cubes $Q \in \cD(Q_0)$ of a fixed length $\ell(Q)$  to observe that, assuming $\|f\|_{\Op(Q_0)}=1$, we have 
	\[
	O(f,Q)\le \ell(Q)^{1-\frac{n}{p}}.
	\]
	This implies that for any sequence of nested dyadic cubes $Q_1\supsetneq Q_2\supsetneq Q_3\supsetneq\ldots$ we have
	\[
	\sum_{k=1}^{\infty}O(f,Q_k)\le \sum_{k=1}^{\infty}\ell(Q_k)^{1-\frac{n}{p}}\le\sum_{i=0}^{\infty}2^{-i(1-\frac{n}{p})}\lesssim 1,
	\]
and an application of Lemma \ref{lerner} finishes the proof in this case.
\end{proof}

\section{Generalized weak-type spaces}
\label{Section2}

In this section we generalize the weak-type quasi-norms defined in \eqref{Opgamma}, and see how they relate to known strong-type norms. 

We start with some notation and assumptions which will hold for the remainder of the paper.  While we already introduced  the notation $\cD$ for the collection of all dyadic cubes in $\rn$, since we will be working with more general measures than Lebesgue measure, we need to specify that these cubes are taken to be half-open, namely of the form
$$Q_{k,j} = [j_12^k, (j_1+1)2^k)\times \ldots \times [j_n 2^k, (j_n+1)2^{k}), \quad k \in \Z, (j_1, \ldots, j_n) \in \Zn,$$
so that the cubes forming the grid at generation $k$ are pairwise disjoint.

Here and in the rest of the paper we assume $\mu$ is a Radon measure on $\rn$. This guarantees that cubes and balls have finite measure and the Lebesgue differentiation theorem holds.
For $1\le p<\infty$ and a real number $\gamma$, we define $l_{\gamma}^{p,\infty}(\rn,\mu)$ as the space of all sequences $a=\left\{a_Q\right\}_{Q\in \cD}$ such that 
\begin{equation}
\label{averagenorm}
	\|a\|_{l_\gamma^{p,\infty}(\rn, \mu)}:=\left(\sup_{\lambda>0}\lambda^p\sum_{|a_Q|>\lambda}\ell(Q)^{-\gamma}\mu(Q)\right)^{\frac{1}{p}} < \infty.
\end{equation}
We suppress $\mu$ in the notation when $\mu(Q) = |Q|$ is just Lebesgue measure.  We will sometimes, but only when explicitly specified, assume that $\mu$ is {\em doubling}, which means that there exists a constant $C_D$ such that
\begin{equation}
\label{doubling}
0 < \mu(2B) \le C_D \mu(B)
\end{equation} 
for all balls $B$, where $2B$ denotes the ball with the same center as $B$ and double the radius.

In Section~\ref{Section1} we considered a special case of the quasi-norm $\|a\|_{l_\gamma^{p,\infty}(\rn, \mu)}$ where the sequence $a$ was given by the mean oscillation of a function $f$ on dyadic cubes, $\mu$ was Lebesgue measure, and $\gamma = p$, namely
$$\|f\|_{\Op(\rn)} =  \|O(f)\|_{l_p^{p,\infty}(\rn)}, \quad O(f) := \left\{O(f,Q)\right\}_{Q\in \cD}.$$
We now want to consider other choices of $a$, $\mu$ and $\gamma$.  In particular, in addition to $a$ defined using mean oscillation, we also want to consider $a$ defined by the means of $f$ on dyadic cubes.

Our goal is to relate these norms to the corresponding strong norms.  Before doing this, we need a couple of lemmas. The following lemma shows that under certain restrictions on $\gamma$, the quasi-norm $\|a\|_{l_\gamma^{p,\infty}(\rn, \mu)}$ can be controlled by only considering disjoint collections of cubes.

\begin{lemma} 
\label{Prop1}
Let $a=\left\{a_Q\right\}_{Q\in \cD}$ a sequence of numbers indexed by $\cD$ and $1\le p<\infty$.  Suppose either
\begin{itemize}
	\item $\gamma<0$, or
	\item $\gamma>d>0$ and $\mu$ satisfies the additional assumption that $\mu(Q) \approx \ell(Q)^d$ for all $Q \in \cD$.
\end{itemize}
Then
\begin{equation}
	\|a\|_{l_\gamma^{p,\infty}(\rn, \mu)} \lesssim \left(\sup_{\lambda>0}\lambda^p\ \sup_\cP \sum_{Q \in \cP, |a_Q|>\lambda}\ell(Q)^{-\gamma}\mu(Q)\right)^{\frac{1}{p}},
\end{equation}
where the supremum on the right is taken over all disjoint subcollections $\cP \subset \cD$.
\end{lemma}

\begin{proof}
Let $\cC$ be a finite collection of dyadic cubes.  Then we can associate to $\cC$ two subcollections $\Cmax$ and $\Cmin$.  Here $\Cmax$ is the collection of all maximal cubes in $\cC$, namely those cubes which are not contained in any larger element of $\cC$.  Similarly, $\Cmin$ is the collection of all minimal cubes, which are those for which none of their subcubes is in $\cC$.  By the properties of the dyadic system, each of these collections consists of disjoint cubes.

\noindent
Case 1:  $\gamma<0.$\\
		 Note that every cube in $\cC$ is contained in some maximal cube. Therefore, denoting by $\cD_k(Q)$ the collection of $k$-th generation subcubes of $Q$, we have
	\begin{align*}
		&\sum_{Q\in {\cC}}\ell(Q)^{-\gamma}\mu(Q)
		\le \sum_{Q\in \Cmax}\sum_{k=0}^{\infty}\sum_{Q'\in \cD_k(Q)}\ell(Q')^{-\gamma}\mu(Q') \\ &\le \sum_{Q\in \Cmax}\mu(Q)\ell(Q)^{-\gamma}\sum_{k=0}^{\infty}2^{k\gamma}\lesssim\sum_{Q\in \Cmax}\mu(Q)\ell(Q)^{-\gamma}.
	\end{align*}

\noindent
Case 2: $\gamma>d>0$ and $\mu(Q) \approx \ell(Q)^d$ for all $Q \in \cD$.\\
 Note that each cube in $\cC$ contains a minimal cube in $\Cmin$ and thus
	\begin{align*}
		& \sum_{Q\in {\cC}}\ell(Q)^{-\gamma}\mu(Q) \lesssim
	\sum_{Q\in {\cC}}\ell(Q)^{d-\gamma} \le \sum_{Q\in \Cmin}\sum_{Q' \in \cC, Q\subset Q'}\ell(Q')^{d-\gamma} \\
	&\le \sum_{Q\in \Cmin}\ell(Q)^{d-\gamma}\sum_{k=0}^{\infty}2^{k(d-\gamma)}\lesssim\sum_{Q\in \Cmin}\mu(Q)\ell(Q)^{-\gamma}.
	\end{align*}

Now fixing $\lambda > 0$ and letting $\cC$ be an arbitrary finite sub-collection of dyadic cubes $Q$ such that $|a_Q|>\lambda$, we see that in the two cases above we get
$$\lambda^p \sum_{Q\in {\cC}}\ell(Q)^{-\gamma}\mu(Q)\lesssim \lambda^p \sup_\cP \sum_{Q \in \cP, |a_Q|>\lambda}\ell(Q)^{-\gamma}\mu(Q).$$
Taking the supremum on the left-hand side over all finite subcollections $\cC$ and over $\lambda > 0$ gives the result of the lemma.
\end{proof}

The second lemma will help us deal with some the values of $\gamma$ not covered in Lemma~\ref{Prop1}.

\begin{lemma}
\label{cubecollectionlemma}
	Let $1<q<\infty$, $\gamma \neq 0$ and $\cC$ be finite collection of dyadic cubes in $\rn$. Then 	
\[
\|\sum_{Q\in \cC}\ell(Q)^{-\frac{\gamma}{q}}\chi_Q\|_{L^{q}(\rn,\mu)}\lesssim \left(\sum_{Q\in\cC}\ell(Q)^{-\gamma}\mu(Q)\right)^{\frac{1}{q}}.
\]
\end{lemma}

\begin{proof}
	First, we partition the collection $\cC$ into a sequence of generations $\cP_1, \cP_2, \ldots$, each of which is a disjoint collection, such that for $j<k$, each cube in the $k$-th generation, $\cP_k$, is contained in a unique cube in the $j$-th generation, $\cP_j$. To do this, simply let $\cP_1 = \Cmax$, the subcollection of maximal cubes in $\cC$, and suppose we have constructed $\cP_j$ for $j<k$. Then the $k$-th generation $\cP_k$ consists of maximal cubes in $\cC\setminus\cup_{j<k} \cP_j$. Now, let 
	\[
	g(x)=\sum_{Q\in \cC}\ell(Q)^{-\frac{\gamma}{q}}\chi_Q(x),
	\]
	 and partition $\cup\cC$ into the sets $E_1, E_2,\ldots$, where
	\[
	E_k=\cup\cP_k \setminus \cup\cP_{k+1}.
	\]
	Here we note that for each $x\in E_k$ there is a unique sequence of cubes $Q_1\supset Q_2\supset\ldots\supset Q_k$, with the property that $Q_j\in \cP_j$ for $1\le j\le k$, and thus
	\begin{equation}\label{geometricseries}
			g(x)=\sum_{j=1}^{k}\ell(Q_j)^{-\frac{\gamma}{q}}.
	\end{equation}

Depending on the sign of $\gamma$, the terms in the above series decay or increase geometrically. If $\gamma<0$, the exponent is positive so the largest cube dominates and we have $g(x)\lesssim \ell(Q_1)^{-\frac{\gamma}{q}}$, which implies
		\[
		g(x)\lesssim\sum_{Q_1\in \cP_1}\ell(Q_1)^{-\frac{\gamma}{q}}\chi_{Q_1}(x).
		\]
	Taking the $L^{q}(\rn,\mu)$ norm of the above inequality yields
		\[
		\|g\|_{L^{q}(\rn,\mu)}\lesssim\left(\sum_{Q_1\in \cP_1}\ell(Q_1)^{-\gamma}\mu(Q_1)\right)^{\frac{1}{q}}\lesssim \left(\sum_{Q\in \cC}\ell(Q)^{-\gamma}\mu(Q)\right)^{\frac{1}{q}},
		\]
		which proves the claim in this case.\\
		
		Suppose now that $\gamma>0$. Then from \eqref{geometricseries} we get $g(x)\lesssim \ell(Q_k)^{-\frac{\gamma}{q}}$, and thus 
		\[
		g(x)\lesssim \sum_{Q_k\in \cP_k}\ell(Q_k)^{-\frac{\gamma}{q}}\chi_{Q_k}(x), \quad \quad x\in E_k.
		\]
	Raising both sides to the power $q$ and integrating the above inequality over $E_k$ gives us
		\[
		\int_{E_k}g^{q}d\mu\lesssim\sum_{Q_k\in \cP_k}\ell(Q_k)^{-{\gamma}}\mu(Q_k).
		\]
	Finally, we note that the support of $g$ is the disjoint union of the sets $E_k$, and by summing the previous inequality over $k$, we obtain
		\[
		\|g\|_{L^{q}(\rn,\mu)}=\left(\sum_{k}\int_{E_k}g^{q}d\mu\right)^{\frac{1}{q}}\lesssim\left(\sum_{Q\in \cC}\ell(Q)^{-{\gamma}}\mu(Q)\right)^{\frac{1}{q}},
		\]
		which completes the proof.
\end{proof}

We are now ready to state and prove our embeddings.
\begin{theorem}
\label{embeddings}
Let $f\in \Loneloc(\rn, \mu)$, $1 \le p < \infty$. Set
$$a(f)_Q:= \ell(Q)^{\frac{\gamma}{p}}\fint_{Q} f d\mu = \ell(Q)^{\frac{\gamma}{p}}\frac{1}{\mu(Q)}\int_{Q} f d\mu \quad \mbox{if } \mu(Q) > 0, \quad a(f)_Q = 0 \mbox{ otherwise}$$
and
$$b(f)_Q:= \ell(Q)^{\frac{\gamma}{p}}O(f,Q).$$

\noindent
{\rm (i)} 
For $\mu$, $p$ and $\gamma$ satisfying the hypotheses of Lemma~\ref{Prop1}, or the hypotheses of Lemma~\ref{cubecollectionlemma},
we have
\begin{equation}
\label{thefirst}
		\|a(f)\|_{l_\gamma^{p,\infty}(\rn, \mu)}\lesssim \|f\|_{\Lp(\rn,\mu)}.		
	\end{equation}

\noindent
{\rm (ii)}If $\mu$ is Lebesgue measure and $\gamma \in \R\setminus[0,n]$ 
then
\begin{equation}\label{thesecond}
		\|f\|_{\Opgamma(\rn)}:=  \|b(f)\|_{l_\gamma^{p,\infty}(\rn)}\lesssim \|f\|_{\JNp(\rn,\cD)}.		
	\end{equation}
	
\noindent
{\rm (iii)} Finally, if $f \in \dot{W}^{1,p}(\rn)$, $1 < p < \infty$, then
\begin{equation}
\label{thethird}
\|f\|_{\Op(\rn)} \lesssim \lpn{\nabla f}.
\end{equation}
\end{theorem}

\begin{proof}
We first consider the cases where we can apply Lemma~\ref{Prop1}, namely $\gamma<0$, or
 $\gamma>d>0$ with $\mu$ satisfying the additional assumption that $\mu(Q) \approx \ell(Q)^d$ for all $Q \in \cD$.

For a disjoint collection  $\cP$ of dyadic cubes $Q$ with $|a(f)_Q| > \lambda$, we have, by Jensen's inequality, 
	\begin{equation*}
\lambda^p	\sum_{Q\in \cP}\ell(Q)^{-\gamma}\mu(Q) \le \sum_{Q\in \cP}\ell(Q)^{-\gamma}\mu(Q)|a(f)_Q|^p \le  \sum_{Q\in \cP}\int_{Q} |f|^p d\mu \le \lpn{f}.
	\end{equation*}
Taking the supremum over all such collections $\cP$ and all $\lambda > 0$ and applying Lemma~\ref{Prop1} gives the result.

Similarly, if $\mu$ is Lebesgue measure, $\gamma \in \R\setminus[0,n]$, and $\cP$ is a disjoint collection of dyadic cubes $Q$ with $|b(f)_Q| > \lambda$, then
	\begin{equation*}
\lambda^p	\sum_{Q\in \cP}\ell(Q)^{-\gamma}|Q| \le \sum_{Q\in \cP}\ell(Q)^{-\gamma}|Q| |b(f)_Q|^p \le  \sum_{Q\in \cP}|Q|O(f,Q)^p \le \|f\|_{\JNp(\Rn, \cD)}.
	\end{equation*}
Again we take the supremum and apply Lemma~\ref{Prop1}.

We can only apply Lemma~\ref{Prop1} to prove \eqref{thethird} for $f \in \dot{W}^{1,p}(\rn)$ in the case $\gamma = p > n$.  We can then use the Poincar\'e inequality
\begin{equation}
\label{Poincare}
O(f,Q) \lesssim \ell(Q)\fint_Q |\nabla f|
\end{equation}
and Jensen's inequality to obtain, as in the proof of \eqref{thefirst} above, that 
$$\|f\|_{\Op(\rn)} \le \sup_\cP \sum_{Q\in \cP}\ell(Q)^{-p}|Q| |O(f,Q)|^p \le \sup_\cP \sum_{Q\in \cP}\int_{Q} |\nabla f|^p = \lpn{\nabla f}.$$

We now use Lemma~\ref{cubecollectionlemma} to prove \eqref{thefirst} in the case $1<p<\infty$ and $\gamma > 0$.

Fix $\lambda>0$ and take an arbitrary finite collection $\cC$ of dyadic cubes $Q$ such that $|a(f)_Q|>\lambda$. By the definition of $a(f)$  we have, for $p'$ the conjugate exponent of $p$,
	\[
	\ell(Q)^{-\gamma}\mu(Q)<\frac{1}{\lambda}\int |f| \ell(Q)^{-\frac{\gamma}{p'}}\chi_Qd\mu, \quad Q\in\cC.
	\]
	Summing over $\cC$ gives us
	\[
\sum_{Q\in\cC}	\ell(Q)^{-\gamma}\mu(Q)<\frac{1}{\lambda}\int|f|\sum_{Q\in \cC}\ell(Q)^{-\frac{\gamma}{p'}}\chi_Qd\mu, 
	\]
	which after applying H\"older's inequality yields
	\[
	\sum_{Q\in\cC}	\ell(Q)^{-\gamma}\mu(Q)<\frac{1}{\lambda}\|f\|_{L^p(\rn,\mu)}\Big\|\sum_{Q\in \cC}\ell(Q)^{-\frac{\gamma}{p'}}\chi_Q\Big\|_{L^{p'}(\rn,\mu)}.
	\]
	An application of Lemma~\ref{cubecollectionlemma}  with $q = p'$ gives
	\[
		\sum_{Q\in\cC}	\ell(Q)^{-\gamma}\mu(Q)\lesssim\frac{1}{\lambda}\|f\|_{L^p(\rn,\mu)}\left(\sum_{Q\in\cC} \ell(Q)^{-\gamma}\mu(Q)\right)^{\frac{1}{p'}},
	\]
	and noting that $\cC$ is finite shows that
	\[
	\left(\lambda^p	\sum_{Q\in\cC}	\ell(Q)^{-\gamma}\mu(Q)\right)^{\frac{1}{p}}\lesssim\|f\|_{L^p(\rn,\mu)}.
	\]
	To finish the proof in this case it is enough to take the supremum first over $\cC$ and then over $\lambda>0$.
	
The proof of \eqref{thethird} for the general the case $1<p<\infty$ is very similar.  Here we again use the Poincar\'e inequality \eqref{Poincare} to get that if $|O(f,Q)|>\lambda > 0$ then
$$\ell(Q)^{-p}|Q|\lesssim\frac{1}{\lambda}\int |\nabla f| \ell(Q)^{1-p}\chi_Q .$$
Noting that in this case $\gamma = p$ means $1 - p =  -\frac{\gamma}{p'}$, the same steps as above and Lemma~\ref{cubecollectionlemma} with $q = p'$ yield 
$$\left(\lambda^p	\sum_{Q\in\cC}	\ell(Q)^{-p}|Q|\right)^{\frac{1}{p}}\lesssim\|\nabla f\|_{L^p(\rn)},$$
thus 	giving \eqref{thethird} .
\end{proof}

The theorem raises several questions.  First, what happens in the cases that are not covered?  As we will see in Section~\ref{Examples}, when $\mu$ is Lebesgue measure and $\gamma \in [0,n]$, the embedding given by \eqref{thefirst} is false for $p = 1$.  As for  the embedding given by \eqref{thesecond}, it fails when $1 \le p < \infty$ and $\gamma \in [0,n]$, and is proper for the other values of $\gamma$.

Second, do the reverse inclusions hold?  We will answer this in the positive for \eqref{thefirst} in Section~\ref{Reverse} below, and in the negative for \eqref{thesecond} via examples in  Section~\ref{Examples}.
We already pointed out in Section~\ref{Section1} that the reverse inclusion in  \eqref{thethird} is false since, by the result of \cite{Frank}, \eqref{frankresult}, this would mean
$\|m(f)\|_{L^{p,\infty}(\Rnp, \nu_p)} \lesssim \|f\|_{\Op(\rn)}$, which fails for a simple example.  In Section~\ref{Continuous} we will discuss the relation between the other dyadic weak-type quasi-norms and their continuous counterparts in the upper half-space.

Finally, are there similar results in the multi-parameter setting?  This will be discussed in the next section.

\section{Bi-parameter embeddings}
\label{Biparameter}

In this section we discuss some counterparts of Theorem~\ref{embeddings} part (i) 
in the two parameter setting. By a dyadic rectangle in $\rnm$ we mean the product of two dyadic cubes in $\rn$ and $\Rm$. We write $\cD_n\times \cD_m$ for the collection of all such dyadic rectangles $R=Q_1\times Q_2$, and denote the sidelengths by $\ell_1(R)=\ell(Q_1)$ and $\ell_2(R)=\ell(Q_2)$.  We set
$$\lmax(R) := \max(\ell_1(R), \ell_2(R)), \quad \lmin(R) := \min(\ell_1(R), \ell_2(R)),$$
and define the mean sidelength of $R$ by
\begin{equation}
\label{lab}
\lab(R) := \lmin(R)^\alpha \lmax(R)^\beta, \quad  \alpha + \beta = 1.
\end{equation}
Note that this gives the geometric mean of the sidelengths when $\alpha = \beta = \frac{1}{2}$, and reduces to $\ell(Q)$ when $Q$ is a cube.

We continue with our standing assumptions that the dyadic cubes (and hence rectangles) are half-open and thus pairwise disjoint, and that $\mu$ is a Radon measure, this time on $\rnm$.

\begin{theorem}
\label{thm-biparameter}
	Let $1<p<\infty$, $\gamma>0$, and $f\in L_{loc}^1(\rnm, \mu)$.  For a dyadic rectangle $R$, using the notation in \eqref{lab}, assume  $\beta \le 0$ and set
	\[
	a(f)_R:= \lab(R)^{\frac{\gamma}{p}}\fint_{R}f d\mu. 
	\]
	 Then
	 \[
	\left(\sup_{\lambda>0}\lambda^p\sum_{|a(f)_R|>\lambda}\labp(R)^{-\gamma}\mu(R)\right)^{\frac{1}{p}}\lesssim\lprnmmun{f}
	\]
where $\alpha' + \beta' = 1$, $0 \le \alpha' < 1$.
	\end{theorem}
Inequality \eqref{thefirst} in part (i) of Theorem~\ref{embeddings} can be obtained from this theorem by restricting to cubes. As in the proof of the case $1 < p < \infty$, $\gamma > 0$ of that inequality, we will show the theorem as a consequence of the following lemma, which is a bi-parameter analogue of Lemma~\ref{cubecollectionlemma}.

\begin{lemma}\label{minmaxlemmarec}
	Let $\cC$ be a finite collection of dyadic rectangles. Using the notation in \eqref{lab}, if  $\gamma > 0$, $1 < q < \infty$, $\delta + \varepsilon = 1$, $\varepsilon > 0$ and  
	\[
	g(x,y)=\sum_{R\in\cC}\lde(R)^{-\frac{\gamma}{q}}\chi_R(x,y),
	\]
	then
	\begin{equation*}\label{recangleestimate}
		\|g\|_{L^{q}(\rnm, \mu)}\lesssim\left(\sum_{R\in\cC}\ldep(R)^{-\gamma}\mu(R)\right)^{\frac{1}{q}}
	\end{equation*}
	for $\delta' + \varepsilon' = 1$, $\delta' \ge 0$, $0 <\varepsilon' < \varepsilon$.	
\end{lemma}

\begin{proof}
	We begin by partitioning $\cC$ into two sub-collections $\cC_1$ and $\cC_2$ as
	\[
	\cC_1=\left\{R\in\cC : \ell_1(R)\le \ell_2(R)\right\}, \quad
	\cC_2=\left\{R\in\cC : \ell_1(R)> \ell_2(R)\right\},
	\]
	and then we prove inequality \eqref{recangleestimate} for the functions $g$ corresponding to $\cC_1$ and $\cC_2$. Since the proofs of these two inequalities are exactly the same, without loss of generality we may assume that for $R\in \cC$ we have $\ell_1(R)\le \ell_2(R)$.\\
	
	Let $\pi_1(\cC)$ and $\pi_2(\cC)$ be the collection of cubes obtained by projecting rectangles in $\cC$ onto $\rn$ and $\mathbb{R}^m$ respectively. Then for $l=1,2$ and $i\in \N$ let $\cC^l_i$ be the $i$-th generation of $\pi_l(\cC)$ and $E^l_i$ be the associated partitioning of $\cup\pi_l(\cC)$, as constructed in the proof of Lemma \ref{cubecollectionlemma}. Next, we partition $E=\cup\cC$ as
	\[
	E=\cup_{j, k}E_{j, k}, \quad E_{j, k}=\left\{(x,y)\in E : x\in E^1_j, \quad y\in E^2_k\right\},
	\]
	and note that for each $(x,y)\in E_{j,k}$ there exists a unique dyadic rectangle $R\in\cC$ such that 
	\[
	R=Q_1\times Q_2, \quad Q_1\in\cC^1_j, Q_2\in\cC^2_k.
	\]
	If $(x,y)$ belongs to a rectangle $R'=Q'_1\times Q'_2 \in \cC$ then we must have
	\[
	Q_1\subset Q'_1, \quad Q_2\subset Q'_2,
	\]
	as well as
	\[
	\ell(Q_1)\le \ell(Q_2), \quad \ell(Q'_1)\le \ell(Q'_2).
	\]
	Write
	$$\lde(R') = \ell_1(R')^{\delta}\ell_2(R')^{\varepsilon} = \ell_1(R')^{\delta-\delta' + \delta'}\ell_2(R')^{\varepsilon}.$$
	Since $\delta-\delta' = \varepsilon' - \varepsilon < 0$, we have
	$$a := (\delta-\delta')\Big(\frac{-\gamma}{q}\Big) > 0,\quad b := \delta' \Big(\frac{-\gamma}{q}\Big) \le 0, \quad c := \varepsilon \Big(\frac{-\gamma}{q}\Big), \quad \mbox{and }
	a + c = \varepsilon' \Big(\frac{-\gamma}{q}\Big) < 0.$$
	
	The above observations give the following estimate for $g(x,y)$:
	\[
	g(x,y)=\sum\limits_{\substack{R'\in \cC\\ (x,y)\in R}}\lde(R')^{-\frac{\gamma}{q}}
	\le \sum\limits_{2^t \ge \ell_2(R)}\sum_{\ell_1(R)\le 2^s \le \ell_2(R)} 2^{s(a+b)}2^{tc} \lesssim \ell_1(R)^b \sum\limits_{2^t \ge \ell_2(R)}2^{t(a+c)} \lesssim \ldep(R)^{-\frac{\gamma}{q}}.
	\]
	Noting that for each $(x,y)$ the rectangle $R$ is unique and integrating $g$ over $E_{j,k}$ by using the above bound, we obtain
	\[
	\int_{E_{j,k}}g^{q}d\mu\lesssim\sum_{\substack{R\in\cC, R=Q_1\times Q_2\\ Q_1\in C^1_j, Q_2\in C^2_k}}\ldep(R)^{-\gamma}\mu(R).
	\]
	Now it is enough to sum over $(j,k)$ and by doing so we get
	\[
	\int_{\rnm}g^{q}d\mu=\sum_{j,k}\int_{E_{j,k}}g^{q}d\mu\lesssim\sum_{R\in\cC}\ldep(R)^{-\gamma}\mu(R),
	\]
	which completes the proof.
	\end{proof}

\begin{proof}[Proof of Theorem~\ref{thm-biparameter}]
Fix $\lambda>0$ and take a dyadic rectangles $R$ such that $|a(f)_R|>\lambda$. By the definition of $a(f)_R$  we have
	\[
	\labp(R)^{-\gamma}\mu(R) = \labp(R)^{-\gamma}\lab(R)^{\frac{\gamma}{p}}\lab(R)^{-\frac{\gamma}{p}}\mu(R)<\frac{1}{\lambda}\int |f| \labp(R)^{-\gamma}\lab(R)^{\frac{\gamma}{p}}\chi_R d\mu.
	\]
	Write $$\labp(R)^{-\gamma}\lab(R)^{\frac{\gamma}{p}} = \left(\lmin(R)^{\alpha'p' - \alpha\frac{p'}{p}}\lmax(R)^{\beta' p' - \beta \frac{p'}{p}}\right)^{-\frac{\gamma}{p'}} = \lde(R)^{-\frac{\gamma}{p'}},$$
	where $\frac 1 p + \frac1 {p'} = 1$, $\delta = \alpha'p' - \alpha\frac{p'}{p} = (\alpha' - \alpha)p' + \alpha$, $\varepsilon = \beta' p' - \beta \frac{p'}{p} = p'(\beta' - \frac{\beta}{p}) > 0$.
Continuing as in the proof of \eqref{thefirst}, we replace the use of Lemma~\ref{cubecollectionlemma} by an application of Lemma~\ref{minmaxlemmarec} with 
$$q = p',\quad \delta' = \alpha' \ge 0,\quad \mbox{and } 0 < \varepsilon' = \beta' < \beta' p' - \beta \frac{p'}{p}  = \varepsilon.$$
	\end{proof}
\section{Reverse inclusions}
\label{Reverse}
We now turn to finding a converse to \eqref{thefirst}.   To do so, for $a$ defined on the dyadic cubes and $\mu$ as specified at the beginning of Section~\ref{Section2}, we consider quantities which are smaller than
$\|a\|_{l_\gamma^{p,\infty}(\rn, \mu)}$, defined by
\begin{equation}
\label{gampos}
	[a]_{l_\gamma^{p,\infty}(\rn, \mu)}:=
		\left(\liminf\limits_{\lambda\to0^+}\lambda^{p}\sum_{|a_Q|>\lambda}\ell(Q)^{-\gamma}\mu(Q)\right)^{\frac{1}{p}},\quad \gamma>0,
\end{equation}
and
\begin{equation}
\label{gamneg}
	[a]_{l_\gamma^{p,\infty}(\rn, \mu)}:=
	\left(\limsup\limits_{\lambda\to\infty}\lambda^{p}\sum_{|a_Q|>\lambda}\ell(Q)^{-\gamma}\mu(Q)\right)^{\frac{1}{p}},\quad \gamma<0.
\end{equation}
In this section we will only consider the case of the means of a function $f\in L_{loc}^1(\rn, \mu)$:
$$a(f)_Q:=\ell(Q)^{\frac{\gamma}{p}}\fint_{Q} fd\mu.$$
Then we have the following:
\begin{theorem}
\label{cubelptheorem}
Let $1\le p<\infty$.  If either
\begin{itemize}
	\item $\gamma>0$, or
	\item $\gamma<0$ and in addition the measure $\mu$ is doubling,
\end{itemize}
then
\begin{equation}
	\|f\|_{L^p(\rn,\mu)}\lesssim [a(f)]_{l_\gamma^{p,\infty}(\rn, \mu)}.
\end{equation}
\end{theorem}

\begin{proof}
Case 1:  $\gamma>0.$\\
 Fix an arbitrary dyadic cube $Q_0$ and $m\in \N$, and recall that $\cD_m(Q_0)$ denotes the collection of dyadic sub-cubes of $Q_0$ with length $2^{-m}\ell(Q_0)$. Set
\[
 \cD'_m=\left\{Q\in \cD_m(Q_0) : a(f)_Q\ne0 \right\}, \quad \lambda_m'=\min\left\{|a(f)_Q|:Q\in \cD'_m\right\}, \quad \lambda_m=\frac{1}{2}\min\left\{\frac{1}{m},\lambda_m'\right\}.
\]
Taking $Q'\in\cD'_m$, note that $|a(f)_{Q'}|> \lambda_m$ and consider the collection
\begin{equation}
\label{small}
\cC(Q') = \{Q \in \cD(Q'): |a(f)_Q|\le \lambda_m\}.
\end{equation}
Denote by $\cC_m(Q')$ the maximal dyadic cubes in $\cC(Q')$ and let $\cC_m=\cup_{Q'\in\cD'_m}\cC_m(Q')$.  Note that this collection has the following properties:\\

1. For $Q\in \cC_m$ we have $\ell(Q)<2^{-m}\ell(Q_0)$ and $|a(f)_Q|\le \lambda_m$. In addition, denoting by $\Qbar$ the parent of $Q$, we have, by the maximality of $Q$, that  $|a(f)_{\Qbar}|>\lambda_m$.\\

2. We have 
\[
\sum_{Q\in\cC_m}\ell(Q)^{-\gamma}\mu(Q)\le 2^{\gamma n}\sum_{Q\in\cC_m}l(\Qbar)^{-\gamma}\mu(\Qbar)\lesssim \sum_{|a_{Q'}(f)|>\lambda_m}l(Q')^{-\gamma}\mu(Q').
\]

3. The measure $\mu(\cup\cD'_m\setminus\cup\cC_m) = 0$.  To see this, consider $x\in \cup\cD'_m\setminus\cup\cC_m$.  Then for some $Q'\in\cD'_m$,  $x \in Q'\setminus \cup \cC(Q')$, which means, by \eqref{small}, that
there is a sequence of dyadic cubes $\left\{Q_i\right\}$ containing $x$, with $\ell(Q_i)\le 2^{-i}\ell(Q')$, such that 
\[
\ell(Q_i)^{\frac{\gamma}{p}}\Big|\fint_{Q_i}fd\mu\Big|>\lambda_m.
\]
Letting $i \ra \infty$ (recall that $m$ is fixed) we note that $\ell(Q_i)^{\frac{\gamma}{p}}\to 0$ since $\gamma>0$, and by the Lebesgue differentiation theorem $|\fint_{Q_i}fd\mu|\to |f(x)|$ for $\mu$-almost every $x$, so we must have $|f(x)|=\infty$, but since $f$ is locally integrable the measure of the set of such points is zero.\\

Consider the function
\[
f_m:=\sum_{Q\in\cC_m}\chi_Q\Big|\fint_{Q}fd\mu\Big|.
\]
By property $2$ above, 
\[
\int_{Q_0}|f_m|^pd\mu=\sum_{Q\in\cC_m}\Big|\fint_{Q}fd\mu\Big|^p\mu(Q)=\sum_{Q\in\cC_m}|a(f)_Q|^p\ell(Q)^{-\gamma}\mu(Q)\lesssim \lambda_m^{p}\sum_{|a_{Q'}(f)|>\lambda_m}l(Q')^{-\gamma}\mu(Q').
\]

Letting $m \ra \infty$ and using the Lebesgue differentiation theorem, recalling that $\fint_{Q} f= 0$ for $Q \in  \cD_m(Q_0)\setminus\cD'_m$, we get the convergence of
 $f_m$ to $|f|$ on $Q_0\setminus \bigcup_m (\cup\cD'_m\setminus\cup\cC_m)$, hence by property $3$ above for $\mu$-almost every point in $Q_0$. Applying Fatou's lemma, and recalling that by definition $\lambda_m \ra 0$, we obtain 
\[
\int_{Q_0}|f|^pd\mu\le\liminf\limits_{m\to\infty}\int_{Q_0}|f_m|^pd\mu\lesssim\liminf\limits_{\lambda\to0^+} \lambda^{p}\sum_{|a_{Q'}(f)|>\lambda}l(Q')^{-\gamma}\mu(Q').
\]
This completes the proof since $Q_0$ was arbitrary.\\

\noindent
Case 2:  $\gamma<0$ and $\mu$ is doubling.\\ 
The argument is analogous to that in Case 1, but taking increasing instead of decreasing $\lambda$.  Again we fix $Q_0 \in \cD$, $m \in \N$, and set
\[
\lambda_m'=\max\left\{|a(f)_Q|:Q\in \cD_i(Q_0), 0\le i\le m \right\}, \quad \lambda_m=2\max\left\{m,\lambda_m'\right\}.
\]
Let $\cC'_m = \Cmax$  where $\cC = \{Q \in \cD(Q_0): |a(f)_Q|\ > \lambda_m\}.$
By the definition of $\lambda_m$, each of the cubes in $\cC'_m$ belongs to $\cD_k(Q_0)$ for  some $k > m$. 

Letting $\cC''_m$ be maximal cubes of the collection of parents $\{\Qbar : Q \in\cC'_m\}$, we observe that this collection enjoys the following properties:\\

$1'$. For each $Q\in \cC''_m$, we have $\ell(Q)\le2^{-m}\ell(Q_0)$ and $|a(f)_Q|\le\lambda_m$, which follows from the maximality of cubes in $\cC'_m$.\\

$2'$. 
Note that each cube $Q\in\cC''_m$ has a child $Q'\in\cC'_m$, and therefore from this and the doubling property of the measure $\mu$ we have
\[
\sum_{Q\in\cC''_m}\ell(Q)^{-\gamma}\mu(Q)\lesssim \sum_{Q'\in\cC'_m}\ell(Q')^{-\gamma}\mu(Q')\lesssim \sum_{|a_{Q''}(f)|>\lambda_m}\ell(Q'')^{-\gamma}\mu(Q'').
\]

$3'$. On the set $Q_0\setminus\cup\cC''_m$, $f$ is zero for $\mu$-almost every point. Similarly to the previous case, this can be seen by noting that if
$$x\in Q_0\setminus\cup\cC''_m \subset Q_0\setminus\cup\cC'_m = Q_0\setminus\cup\{Q \in \cD(Q_0): |a(f)_Q|\ > \lambda_m\},$$
 then there is a sequence of dyadic cubes $\left\{Q_i\right\}$ containing $x$, with $\ell(Q_i)\le 2^{-i}\ell(Q')$, such that 
\[
l(Q_i)^{\frac{\gamma}{p}}|\fint_{Q_i}fd\mu|\le\lambda_m.
\]
Letting $i \ra \infty$ (with $m$ fixed), since in this case $\gamma<0$, then $l(Q_i)^{\frac{\gamma}{p}}\to\infty$ and therefore we must have $|\fint_{Q_i}fd\mu|\to0$.  Applying the  Lebesgue differentiation theorem gives $f(x)=0$ for $\mu$-almost every point in $Q_0\setminus\cup\cC''_m$.\\

Now we set
$$f_m:=\sum_{Q\in\cC''_m}\chi_Q\Big|\fint_{Q}fd\mu\Big| $$
and let $m \ra \infty$, following the same steps as in Case 1, to complete the proof.
\end{proof}

In the bi-parameter setting, we can define $\|a\|_{l_\gamma^{p,\infty}(\rn, \mu)}$ for a sequence $\{a_R\}$ indexed by dyadic rectangles $R \in \cD_n \times \cD_m$ by replacing $\ell(Q)$ in \eqref{gampos} and \eqref{gamneg} with the  mean sidelength $\lab(R)$ defined in \eqref{lab}.  As we already pointed out,  $\lab(R)$ reduces to $\ell(Q)$ when $R = Q$ is a cube, so the fact that $\cD(\rnm) \subset \cD_n \times \cD_m$ means that Theorem~\ref{cubelptheorem} implies its analogue in the two-parameter setting.
\section{Generalized weak quasi-norms in the continuous setting}
\label{Continuous}
Analogously to the case described in Section~\ref{Section1}, one can define weak-type quasi-norms in the continuous setting, on the upper half-space, and ask what is the relation between the dyadic weak-type quasi-norms we have been considering and their continuous versions.

Assume $\mu$ is doubling, $\gamma \neq 0$, and define a measure on $\Rnp$ by
\[
d\nu_{\gamma}:=\frac{d\mu dt}{t^{1+\gamma}}.
\]
 Fix $1\le p<\infty$.  With an abuse of notation, we use $a$ to denote both a sequence indexed by dyadic cubes, given by
$$a_Q := \ell(Q)^{\frac \gamma p} \omega(Q), \quad Q \in \cD,$$
and a function on the upper half-space, given by
$$a(x,t) := t^{\frac \gamma p} \omega(B(x,t)), \quad (x,t) \in \Rnp,$$
where $\omega$ is nonnegative function defined on cubes and balls in $\Rn$.  In the examples we are considering, $\omega$ will involve either the mean or the mean oscillation of some function $f \in \Loneloc(\Rn)$ on the relevant sets.

We want to relate the dyadic quasi-norms $\|a\|_{l_\gamma^{p,\infty}(\rn, \mu)}$ and $[a]_{l_\gamma^{p,\infty}(\rn, \mu)}$ with their continuous versions $\|a\|_{L^{p,\infty}(\Rnp, \nu_p)}$ and $[a]_{L^p(\Rnp,d\nu_{\gamma})}$.  This amounts to comparing the measures of the corresponding level sets, that is
$$\sum_{Q \in \cD, a_Q >\lambda}\ell(Q)^{-\gamma}\mu(Q) \mbox{ and } \nu_\gamma\left( \left\{(x,t) \in \Rnp: a(x,t) > \lambda\right\}\right).$$

To this aim let $Q$ be a dyadic cube such that $a_Q>\lambda$. Set $\hat{Q}=Q\times[\ell(Q)\sqrt{n}, 2\ell(Q)\sqrt{n}]$ and note that for $(x,t)\in \hat{Q}$ we have $t \approx \ell(Q)$ and, denoting the center of $Q$ by $x_Q$, 
$$B(x_Q, \ell(Q)/2) \subset Q \subset B(x,t) \subset B(x_Q, 5\ell(Q)\sqrt{n}/2) \subset 2^kB(x_Q, \ell(Q)/2)$$ 
for some $k$ depending on $n$.
Thus using the doubling assumption \eqref{doubling} on  $\mu$ we get
	\[
	  \mu(B(x,t))\approx \mu(Q).
	\]
If, under these conditions, $\omega$ satisfies 
\begin{equation}
\label{monotone}
\omega(Q) \lesssim \omega(B(x,t)),
\end{equation}
then we can conclude that for some constant $c_1$ we have
	\[
	a(x,t)>c_1 a_Q>c_1\lambda,\quad (x,t)\in \hat{Q}.
	\]
Next, we note that
\[
\nu_{\gamma}(\hat{Q})\approx \mu(Q)\ell(Q)^{-\gamma}.
\]
Moreover, for two different dyadic cubes $Q_1$ and $Q_2$ ,the sets $\hat{Q}_1$ and $\hat{Q}_2$ are disjoint in the upper half space.  Thus
	\[
	\sum_{a_Q>\lambda}\mu(Q)\ell(Q)^{-\gamma}\approx\sum_{a_Q>\lambda}\nu_{\gamma}(\hat{Q})\le \nu_\gamma\left( \left\{(x,t) \in \Rnp: a(B(x,t)) >c_1\lambda\right\}\right),
	\]
which gives that
\begin{equation}
\label{contidaydicversion}
	\sum_{Q \in \cD, a(Q) >\lambda}\ell(Q)^{-\gamma}\mu(Q) \lesssim \nu_\gamma\left( \left\{(x,t) \in \Rnp: a(x,t) >c_1\lambda\right\}\right).
\end{equation}	

	For the reverse inequality, let  $k\in \Z$, set 
	\[
	E_k=\left\{a>\lambda\right\}\cap\rn\times(2^k,2^{k+1}],\quad 
	\]
	and suppose $F_k$ is the projection of $E_k$ on $\rn$. Note that for each $x\in F_k$ there exists a ball $B(x,t)$ with $t\in(2^k, 2^{k+1}]$ such that $a(x,t)>\lambda$, and thus an application of the Vitali's covering lemma gives us a sub-collection of disjoint balls $\cB_k = \left\{B\right\}$ such that 
	$$F_k\subset \bigcup_{B \in \cB_k} 3B.$$
	From this we obtain
	\[
	\nu_{\gamma}(E_k)\le \nu_{\gamma}(F_k\times(2^k,2^{k+1}])\lesssim 2^{-k\gamma}\sum_{B \in \cB_k}  \mu(B).
		\]
	Cover each $B = B(x,t)\in \cB_k$ by at most $3^n$ dyadic cubes $\{Q_i\}$ of sidelength comparable to $t$.  If under these conditions $\omega$ satisfies
\begin{equation}
\label{subadditive}
\omega(B(x,t)) \lesssim \max_i \omega(Q_i),
\end{equation}	
	then there is a constant $c_2$ such that for at least one of these cubes $Q_i$ we must have $a_{Q_i}>c_2\lambda$.  Denoting this choice by $Q_B$, we have, by doubling,
	$$\nu_{\gamma}(E_k)\lesssim \sum_{B \in \cB_k} \ell(Q_B)^{-\gamma}\mu(Q_B).$$
Since  the balls in $\cB_k$ are disjoint and have radii 
	in $(2^k, 2^{k+1}]$, we can take all the $Q_B$ to have the same sidelength, and there is a uniform bound on $\#\{B \in \cB_k: Q_B = Q\}$ for $Q \in \cD$. As we vary $k$, the sidelengths of the $Q_B$ will be 	different. Summation over $k\in \Z$ thus gives us
\begin{equation}
\label{contidaydicversion2}
	\nu_{\gamma}\left(\left\{a(x,t)>\lambda\right\}\right)
	\lesssim \sum_{Q \in \cD, a_Q>c_2\lambda}\ell(Q)^{-\gamma}\mu(Q).
\end{equation}

Based on the above analysis, we would like to formulate the relationship between the continuous and the dyadic weak-type quasi-norms in some special cases. 
In order to do so, we need to replace the classical system of dyadic cubes, $\mathcal{D}$, by a more general notion of a dyadic system of cubes defined as follows.
 
 \begin{definition}
 	A dyadic system of cubes $\mathcal{D'}$ in $\rn$ is a collection of cubes with the following properties:
 	\begin{itemize}
 		\item Every two cubes in $\mathcal{D'}$ are either disjoint or nested.
 		\item If $Q\in \mathcal{D'}$, then the children of $Q$ (the cubes obtained by bisecting its sides) belong to $\mathcal{D'}$.
 		\item Any cube in $\mathcal{D'}$ is a child of another cube in $\mathcal{D'}$, called its parent.
 		\item Cubes of a fixed length partition the entire space $\rn$.	
 	\end{itemize}
 \end{definition}
In the results proven up to now,  these are the only properties of the classical dyadic system $\cD$ that we have used. Therefore, for any dyadic system $\mathcal{D'}$, a sequence of numbers $a=\{a_Q\}_{Q\in\mathcal{D'}}$, and a function $f$, we can define the norms $[a]_{l_\gamma^{p,\infty}(\rn, \mu, \mathcal{D'})}$ and $\|f\|_{\Op(\Rn, \mathcal{D'})}$ and obtain the same conclusions.

Such generalized dyadic systems, called {\em dyadic flitrations} in \cite{CondeAlonso} or {\em dyadic lattices}  in \cite{LernerNazarov}, can be taken to have a property which is missing from the classical dyadic system:  the existence of a finite family of dyadic systems $\{\cD_j\}_{j = 1}^m$ and a constant $C = C_{m,n}$ such that
\begin{equation}
\label{eqn-3lattice}
\mbox{for every ball } B, \exists j \in \{1,\ldots, m\} \mbox{ and a cube } Q\in\mathcal{D}_j \mbox{ with }
B\subset Q \mbox{ and } |Q| \le C|B|.
\end{equation}
We refer to \cite{CondeAlonso} for an efficient construction of such family of dyadic systems with $m=n+1$, or to the Three Lattice Theorem \cite[Theorem 3.1]{LernerNazarov} for a different, easier construction with $m=3^n$.

The existence of such a family allows us to state parts (ii) and (iii) of the following result.

\begin{proposition}
\label{prop-contdyadic}
Let $f\in \Loneloc(\rn, \mu)$, $1 \le p < \infty$. 

\noindent
{\rm (i)} Define $a(f)_Q$, $Q \in \cD$, as in Theorem~\ref{embeddings} and set
$$a(f)(x,t):=t^{\frac{\gamma}{p}}\fint_{B(x,t)}f d\mu.$$
If $f$ is nonnegative then
$$\|a(f)\|_{l_\gamma^{p,\infty}(\rn, \mu)} \approx \|a(f)\|_{L^{p,\infty}(\Rnp, \nu_p)}$$
and
$$[a(f)]_{l_\gamma^{p,\infty}(\rn, \mu)} \approx [a(f)]_{L^p(\mathbb{R}_+^{n+1},d\nu_{\gamma})}.$$

\medskip
\noindent
Let  $\mu$ be Lebesgue measure and $\{\cD_j\}_{j = 1}^{n+1}$ be a family of dyadic systems for which \eqref{eqn-3lattice} holds.  
\noindent
{\rm (ii)} Define $b(f)_Q$, $Q \in \cD_j$, as in Theorem~\ref{embeddings}, and set
$$b(f)(x,t):=t^{\frac{\gamma}{p}}O(f,B(x,t)).$$
Then 
\begin{equation}
\label{continousdyadic}
\|f\|_{\Opgamma(\rn)} = \|b(f)\|_{l_\gamma^{p,\infty}(\rn)} \lesssim \|b(f)\|_{L^{p,\infty}(\Rnp, \nu_{\gamma})} \lesssim \max_{j=1,..,n+1}\|b(f)\|_{l_\gamma^{p,\infty}(\rn, \cD_j)}.
\end{equation}

\medskip
\noindent
{\rm (iii)} Let $\gamma = p$. Define $\|f\|_{\Op(\Rn)}$ and $\|m(f)\|_{L^{p,\infty}(\Rnp, \nu_p)}$ as in Section~\ref{Section1}, and $\|f\|_{\Op(\Rn, \cD_j)}$ by analogy with \eqref{Opgamma}.  Then 
$$\|f\|_{\Op(\Rn)} \lesssim \|m(f)\|_{L^{p,\infty}(\Rnp, \nu_p)} \lesssim \max_{j=1,..,n+1}\|f\|_{\Op(\Rn, \cD_j)}.$$
\end{proposition}

\begin{proof}
We begin by proving the inequality \eqref{contidaydicversion} in all three cases.  From the discussion preceding the statement of the proposition, it suffices to verify condition \eqref{monotone} for $\omega$.  Writing $B = B(x,t)$, we verify this in case (i), for $f$ nonnegative, by
$$\omega(Q) :=  \fint_Q f d\mu \le \frac{\mu(B)}{\mu(Q)} \frac{1}{\mu(B)}\int_{B} f d\mu \lesssim \fint_{B} f d\mu = \omega(B).$$
In cases (ii) and (iii) it follows from
$$\omega(Q) := O(f,Q) \le 2\fint_Q |f - f_B| \lesssim \fint_{B} |f  - f_B| = O(f,B) = \omega(B)$$
and
$$\omega(Q) := \ell(Q)^{-1}O(f,Q) \le 2\ell(Q)^{-(n+1)}\int_Q |f - f_B| \lesssim t^{-1}\fint_{B} |f  - f_B| = \omega(B).$$

For the reverse inequality, \eqref{contidaydicversion2}, the condition that is required on $\omega$ is \eqref{subadditive}.  Again this is verified in case (i) by writing
$$\fint_B f d\mu \le  \frac{1}{\mu(B)}\sum_i \mu(Q_i) \fint_{Q_i} f d\mu \le \frac{\mu(\cup Q_i)}{\mu(B)} \max_i \fint_{Q_i} f d\mu \lesssim \max_i \fint_{Q_i} f d\mu,$$
where we have used the doubling property of $\mu$ in the last inequality.

In cases (ii) and (iii), condition \eqref{subadditive} is no longer true with $\omega$ involving the mean oscillation of $f$, since when using $Q_i$ in a single dyadic system, it is possible to have $O(f, B) > 0$ but $O(f,Q_i) = 0$ for all $i$.
However, if we take $\{\cD_j\}_{j = 1}^{n+1}$ to be the family of dyadic systems in the hypothesis of the proposition, then 
for each ball $B$, letting $Q_B$ be a cube in one these systems for which \eqref{eqn-3lattice} holds, we have 
\[
O(f,B)\lesssim O(f,Q_B)
\]

and we can continue as in the discussion preceding the statement of the proposition to obtain, instead of \eqref{contidaydicversion2}, that
$$\nu_{\gamma}\left(\left\{a(x,t)>\lambda\right\}\right)
	\lesssim \max_{j=1,\ldots n+1} \sum_{Q \in \cD_j, a_Q>c_2\lambda}\ell(Q)^{-\gamma}\mu(Q).$$
	
The same arguments will work for part (iii).
\end{proof}

Combining the results of Proposition~\ref{prop-contdyadic} with those of Theorem~\ref{embeddings} and Theorem~\ref{cubelptheorem}, we obtain the following results.  
\begin{corollary}
\label{contembeddings}
Let $f$, $a(f)$, $b(f)$ and $p$ be as in the statement of Proposition~\ref{prop-contdyadic}.

\noindent
{\rm (i)} Assume $f$ is nonnegative and 
\begin{itemize}
	\item $\gamma < 0$, or
	\item $\gamma>d>0$ and $\mu$ satisfies the additional assumption that $\mu(Q) \approx \ell(Q)^d$ for all $Q \in \cD$, or
	\item $\gamma > 0$ and $p > 1$.
\end{itemize}
Then
		$$\|a(f)\|_{L^{p,\infty}(\Rnp, \nu_\gamma)}\lesssim \|f\|_{\Lp(\rn,\mu)}.$$
Conversely, if
\begin{itemize}
	\item $\gamma>0$ or
	\item $\mu$ is doubling and $\gamma<0$,
\end{itemize}	
then	
$$\|f\|_{\Lp(\rn,\mu)} \lesssim  [a(f)]_{L^{p,\infty}(\Rnp, \nu_\gamma)}.$$	

\noindent
{\rm (ii)} Assume $\mu$ is Lebesgue measure and $\gamma\notin [0,n]$. Then 
	\begin{equation}
	\label{jnp}
	\|b(f)\|_{L^{p,\infty}(\Rnp, \nu_\gamma)}
	\lesssim \|f\|_{\JNp(\rn)}.
	\end{equation}
	
\noindent
{\rm (iii)} If $f \in \dot{W}^{1,p}(\rn)$, $1 < p < \infty$, then
$$\|m(f)\|_{L^{p,\infty}(\Rnp, \nu_p)} \lesssim \lpn{\nabla f}.$$
\end{corollary}

\begin{proof}
Part (i) follows immediately from the proposition and the theorems.  To prove \eqref{jnp}, by part (ii) of the proposition
$$\|b(f)\|_{L^{p,\infty}(\rn, \nu_{\gamma})} \lesssim \max_{j=1,..,n+1}\|b(f)\|_{l_\gamma^{p,\infty}(\rn, \cD_j)},$$
and applying the result of Theorem~\ref{embeddings} to each dyadic system separately, we get
$$
\|b(f)\|_{l_\gamma^{p,\infty}(\rn, \cD_j)}\lesssim \|f\|_{\JNp(\rn,\cD_j)}.$$
Finally, we observe that $\|f\|_{\JNp(\rn,\cD_j)} \leq  \|f\|_{\JNp(\rn)}$.

The inequality in part (iii) is proved in exactly the same way, replacing $\cD$ by $\cD_j$ in Theorem~\ref{embeddings} to get
$$\|f\|_{\Op(\Rn, \cD_j)} \lesssim \lpn{\nabla f}.$$
\end{proof}

\begin{remark} Note that part (iii) of Corollary~\ref{contembeddings} reproduces one side of the inequality \eqref{frankresult} from \cite[Theorem 1]{Frank}.

A direct proof of part (ii) of Corollary~\ref{contembeddings}, not going through dyadic arguments, can be obtained by simply replacing balls by cubes and the averages $T_t f(x) = \fint_{B(x,t)} f$ by the oscillation $O(f,Q)$ in the proof of the weak-type inequality in \cite[Theorem 2.10]{Oscar}.
\end{remark}
\section{Examples}
\label{Examples}

In this section we present some examples to show that the embeddings proved in the previous sections are proper or may fail for the cases which are not covered.

\begin{claim}
\label{Claim1}
For any $\gamma \in [0,1]$ and $n = p = 1$, inequality \eqref{thefirst} fails to hold.
\end{claim}

\setcounter{example}{-1}
\begin{example}
\label{Example 0}
\end{example} For the sake of completeness, we note that for $\gamma = 0$, any $f$ with $f > \lambda$ on some $J \in \cD$ for some $\lambda > 0$ gives
$$\|a(f)\|_{l_1^{1,\infty}[0,1]} \ge \lambda^p \sum_{|a(f)_I|>\lambda}\ell(I) = \lambda^p \sum_{|f_I|>\lambda}\ell(I) \ge   \lambda^p \sum_{I \in \cD(J)}\ell(I) = \infty.$$
The norm can be shown to be infinite also for any $f$ which is nonzero on a set of positive measure.

\begin{example}
\label{Example 1}
\end{example} Let $\gamma=1= p$ and consider the function
\[
f=\sum_{n=1}^{\infty}\frac{2^{n^3}}{n^2}\chi_{[0,2^{-n^3}]},
\]
which belongs to $L^1[0,1]$. Note that for each $n$, and every dyadic interval $I\subset [0,1]$ which contains $[0,2^{-n^3}]$, we have
\[
a(f)_I = \ell(I)\fint_{I}f\ge\int_{[0,2^{-n^3}]}f>\frac{1}{n^2},
\]
which implies that for $\lambda_n=\frac{1}{n^2}$ we have
\[
\lambda_n\sum_{a(f)_I>\lambda_n}\ell(I)^{1-\gamma}\ge \frac{\#\{I\in \cD([0,1]) : I \supset [0,2^{-n^3}]\}}{n^2} \ge n.
\]
So in this case $ \left[a(f)\right]_{l_1^{1,\infty}[0,1]}=\infty$ and hence $\|a(f)\|_{l_1^{1,\infty}[0,1]} = \infty$.

\begin{example}
\label{Example 2}
\end{example} This time we provide an example of the failure of \eqref{thefirst} in the case $0<\gamma<1$.  
Take a sequence of natural numbers $\left\{n_k\right\}$ such that $2^{n_k(1-\gamma)} \ge 2^k$.  If $\lfloor x \rfloor$ denotes the greatest integer less than or equal to $x$, and
$\{x\} = x - \lfloor x \rfloor$ denotes its fractional part, then we set
\begin{equation}
\label{numbersconditions}
q_k := \lfloor2^{n_k(1-\gamma)}\rfloor, \quad \frac{q_k}{2^{n_k(1-\gamma)}} =1 -\varepsilon_k, \quad \varepsilon_k = \frac{\{2^{n_k(1-\gamma)}\}}{2^{n_k(1-\gamma)}}\le 2^{-k},\quad k\in \N.
\end{equation}

Let $\cC_1$ be a subcollection of  $\cD_{n_1}([0,1])$ containing exactly $q_1$ intervals. Suppose now we have constructed $\cC_i$, then $\cC_{i+1}$ is obtained by selecting exactly $q_{i+1}$ intervals in $\cD_{n_{i+1}}(I)$ for each $I \in \cC_i$. From the construction for $C_k$ we have
\begin{equation}\label{collectionproperties}
	\#\cC_{k}=q_1\cdots q_k, \quad \ell(I)=2^{-(n_1+ \ldots +n_k)}, \quad I\in\cC_k.
\end{equation}

Now, pick $0<\alpha<1$ and let
\[
f:=\sum_{k=1}^{\infty}\frac{2^{\gamma(n_1+ \ldots +n_k)}}{k^{1+\alpha}}\sum_{I\in\cC_k}\chi_I.
\]
From \eqref{numbersconditions} and \eqref{collectionproperties} it follows that 
\[
|\cup\cC_k|=q_1\cdots q_k2^{-(n_1+ \ldots +n_k)} \le 2^{-\gamma(n_1+ \ldots +n_k)},
\]
which implies that $f\in L^1[0,1]$. 

Next, we fix $k \in \N$, $I\in \cC_k$, and estimate $a(f)_I$ from below. To do this we note that for each natural number $i$ the total number of intervals of $\cC_{k+i}$ inside $I$ is $q_{k+1}\cdots q_{k+i}$. From this observation and \eqref{numbersconditions} it follows that
\begin{align*}
	a(f)_I&=\ell(I)^{\gamma}\fint_{I}f\ge 2^{-(n_1+ \ldots +n_k)(\gamma-1)}\sum_{i=1}^{\infty}\frac{q_{k+1}\cdots q_{k+i}}{(k+i)^{1+\alpha}}2^{-(n_1+ \ldots +n_{k+i})(1-\gamma)}\\
	&=\sum_{i=1}^{\infty}\frac{(1-\varepsilon_{k+1})\cdots(1-\varepsilon_{k+i})}{(k+i)^{1+\alpha}} \ge \prod_{l=1}^{\infty}(1-2^{-l})\sum_{i=1}^{\infty}\frac{1}{(k+i)^{1+\alpha}}.
	\end{align*}
Denoting half of the quantity on the right-hand side by $\lambda_k$, we have that $\lambda_k \approx k^{-\alpha}$, and for $1 \le j \le k$ and $I\in\cC_j$, $a(f)_I > \lambda_j > \lambda_k$, so 
\begin{align*}
	\lambda_k\sum_{a(f)_I>\lambda_k}\ell(I)^{1-\gamma} &\gtrsim k^{-\alpha}\sum_{j=1}^{k}\sum_{I\in\cC_j}\ell(I)^{1-\gamma}
	=k^{-\alpha}\sum_{j=1}^{k}q_1\cdots q_j 2^{-(n_1+ \ldots +n_j)(1-\gamma)}\\
	&=k^{-\alpha}\sum_{j=1}^{k}(1-\varepsilon_1)\cdots (1-\varepsilon_k) \ge k^{-\alpha}\sum_{j=1}^{k}(1-2^{-1})\cdots (1-2^{-k})\gtrsim k^{1-\alpha}.
\end{align*}
Again this shows that $\left[a(f)\right]_{l_\gamma^{1,\infty}[0,1]}=\infty$.

\begin{claim}
\label{Claim2}
The inclusion given by \eqref{thesecond}  is proper, i.e.\ $\|f\|_{\JNp(\Rn,\cD)} \lesssim \|b(f)\|_{l_\gamma^{p,\infty}(\Rn)}$ fails to hold.  Moreover, we cannot replace $\|f\|_{\JNp(\Rn,\cD)}$ in \eqref{thesecond} by $\|f\|_{L^{p,\infty}}$, and furthermore the norm $\|b(f)\|_{l_\gamma^{p,\infty}(\Rn)}$ is not invariant under rearrangements.
\end{claim}

\begin{example}
\label{Example 3}
\end{example} Let $\gamma\ne 1$, $p>1$, and consider the function $f(x)=x^{-\frac{1}{p}}\chi_{[0,1]}$. In \cite[Example 3.5]{ABKY} this is given as an example of a function which is in weak-$L^p$ but not in $\JNp([0,1])$, but the calculation there actually shows it is not in $\JNp([0,1],\cD)$.  

We want to show that $\|b(f)\|_{l_\gamma^{p,\infty}} < \infty$. Let $I=[j2^{-k},(j+1)2^{-k}) \in \cD_k([0,1))$. If $j=0$ we use the trivial bound $O(f, I) \le 2f_I$ to get
\[
b_I(f) = \ell(I)^{\frac{\gamma}{p}}O(f, I)  \lesssim 2^{-k\frac{\gamma}{p}} \fint_{[0,2^{-k}]}x^{-\frac{1}{p}}dx\lesssim 2^{k\frac{1-\gamma}{p}},
\]
and if $j\ge1$ the Poincar\'e inequality gives us
\[
b_I(f)\lesssim 2^{-k(1+\frac{\gamma}{p})}\fint_I x^{-(1+\frac{1}{p})}dx\lesssim 2^{k\frac{1-\gamma}{p}}j^{-(1+\frac{1}{p})}.
\]
So for $\lambda>0$ we have $b_I(f)>\lambda$ only if $2^{k\frac{1-\gamma}{p}}(j+1)^{-(1+\frac{1}{p})}\gtrsim \lambda$, which implies that
\[
\sum_{b(f)_Q>\lambda}\ell(I)^{1-\gamma}\lesssim \sum_{j=0}^{\infty}\sum_{\substack{2^{k(\gamma-1)}\lesssim \\ \lambda^{-p}(j+1)^{-(p+1)}}}2^{k(\gamma-1)}\lesssim \lambda^{-p}\sum_{j=0}^{\infty}(j+1)^{-(p+1)}\lesssim \lambda^{-p}.
\]
This shows that $\|b(f)\|_{l_\gamma^{p,\infty}(\R)} < \infty$.

Setting $F(x_1, \ldots, x_n)=f(x_1)$ and using the fact that for $Q\in \cD(\Rn)$, $O(F,Q)  =  O(f,I)$, where $I$ is the projection of $Q$ on the $x_1$ axis, we obtain an example of a function $F$ for which $\|b(F)\|_{l_\gamma^{p,\infty}(\Rn)} < \infty$, but which is not in $\JNp(\Rn,\cD)$, for $p>1$.

The function $f$ above can be modified to show that  \eqref{thesecond} fails with $\|f\|_{\JNp(\Rn,\cD)}$ replaced by $\|f\|_{L^{p,\infty}(\Rn)}$, and also to show that $\|b(f)\|_{l_\gamma^{p,\infty}(\Rn)}$ is not invariant under rearrangements. 
Let $\gamma\ne 1,0$, and take a large number $k>2$. Then, cut the interval $[0,1]$ into pieces of length $2^{-k}$, and consider the function $f$ defined by
\[
f_k(x)=x^{-\frac{1}{p}}\sum_{m=0}^{2^k-1} (-1)^{m} \chi_\Ikm, \quad \Ikm = [m2^{-k}, (m+1)2^{-k}).
\]
Note that $|f_k(x)| = f(x) = x^{-\frac{1}{p}}\chi_{[0,1]}$ so they have the same distribution function and $\|f_k\|_{L^{p,\infty}(\Rn)} = \|f\|_{L^{p,\infty}(\Rn)}$.  However, as we show below, $\|b(f_k)\|_{l_\gamma^{p,\infty}(\Rn)}$ blows up with $k$.

For any dyadic interval $\Ijm$ with $1\le m<2^{j}$, and $1\le j<k$, we have that on one half of $I$, $f_k \approx 2^{\frac{j}{p}}m^{-\frac{1}{p}}$, and on the other half $f_k$ has the same size but with the opposite sign, which implies that
\[
b_I(f_k)= \ell(I)^{\frac{\gamma}{p}}O(f,I) \approx 2^{\frac{j}{p}(1-\gamma)}m^{-\frac{1}{p}}.
\]

Set 
$$\lambda_k^p = \left\{  \begin{array}{cc}
        \max (1, 2^{-\gamma(k-2)}) &  \gamma<1,\\
        2^{(1-\gamma)(k-2)}  & \gamma>1
    \end{array}
    \right.
\quad 
\mbox {and}
\quad
N_k = \left\{  \begin{array}{cc}
        2^{k-2}&  \gamma<0,\\
        2^{|\gamma-1|(k-2)} & \gamma>0.
    \end{array}
    \right.
$$

Next, note that in all theses cases, for $1<m<N_k$, we may choose $j$ so that
\[
0\le j\le k-1, \quad m<2^j, \quad 2^{j-1}\le (m\lambda_k^p)^{\frac{1}{1-\gamma}}<2^{j},
\] 
and therefore
\[
b_\Ijm(f_k)\approx  \lambda_k, \quad \ell(\Ijm)^{1-\gamma}\approx  \lambda_k^{-p}m^{-1}.
\]
Thus
\[
\|b(f_k)\|^p_{l_\gamma^{p,\infty}(\Rn)} \ge \lambda_k^p\sum_{b_I(f_k)>\lambda_k} \ell(I)^{1-\gamma}\gtrsim \sum_{1 < m < N_k}m^{-1}\approx  \log N_k.
\]
As $k$ increases, $N_k$ grows without bound, so that $\|b(f_k)\|_{l_\gamma^{p,\infty}(\Rn)}$ cannot be bounded by $\|f_k\|_{L^{p,\infty}(\Rn)} = \|f\|_{L^{p,\infty}(\Rn)}$.

\begin{claim}
\label{Claim3}
The inequality \eqref{thesecond} fails when $0<\gamma\le n$, namely $\|b(f)\|_{l_\gamma^{p,\infty}(\Rn)} \lesssim \|f\|_{\JNp(\Rn,\cD)}$ fails to hold, 
and therefore so does the embedding
$\|b(f)\|_{l_\gamma^{p,\infty}(\Rn)} \lesssim \|f\|_{L^{p,\infty}}$.
\end{claim}

\begin{example}
\label{Example 4}
\end{example} Here we take $\gamma=n$ and $1<p<\infty$. For $k \in \Z$, let $Q_k=[0,2^{-k}]^n$ and consider the function
\[
f=\sum_{k=1}^{\infty}c_k\chi_{Q_k}, \quad c_k = 2^{k\frac{n}{p}}.
\] 
Then
$$\int f =  \sum_{k=1}^{\infty}c_k|Q_k| = \sum_{k = 1}^{\infty}2^{k(\frac{n}{p}-n)} < \infty.$$

First we show that $f$ is in $\JNp(\Rn,\cD)$.  Note that for any disjoint collection of dyadic cubes there is at most one cube on which $f$ is not constant, and that cube must be $Q_k$ for some $k \in \Z$. Therefore it is enough to show that 
\begin{equation}
\label{jnpconditionexample4}
	|Q_k|O(f,Q_k)^p\lesssim1,\quad k \in \Z.
\end{equation}
The trivial estimate of the oscillation by the mean gives  
\begin{equation}
\label{largecubes}
|Q_k|O(f,Q_k)^p \le 2^{-kn(1 - p)+p} \|f\|^p_{\Lone} \lesssim 1, \quad  k \le 0.
\end{equation}
For $k \ge 1$, applying the same estimate to the function $g = f - c_k$, we have that
\begin{equation}
\label{oscillationabove}
O(f,Q_k) = O(g, Q_k) \le 2\fint_{Q_k} \sum_{l=k+1}^{\infty}c_l \chi_{Q_l} = 2^{kn+1}\sum_{l=k+1}^{\infty}2^{l(\frac{n}{p}-n)}\approx 2^{k\frac{n}{p}},
\end{equation}
so \eqref{jnpconditionexample4} holds 
and thus $f\in \JNp(\Rn,\cD)$. 

We can use a similar argument to bound the oscillation of $g$ on $Q_k$  from below.  For $k \ge 1$,
\begin{align}
\label{oscillationbelow}
2O(g,Q_k) &\ge  \fint_{Q_k} \fint_{Q_k} |g(x) - g(y)| dx dy \ge \frac{1}{|Q_k|^2} \int_{Q_k\setminus Q_{k+1}} \int_{Q_{k+1}} |g(x)| dx dy \\
&\ge \frac{|Q_k\setminus Q_{k+1}||Q_{k+1}|}{|Q_k|^2} |c_{k+1} - c_k| = (1 - 2^{-n})2^{-n}|2^{(k+1)(\frac{n}{p})} - 2^{k(\frac{n}{p})}| \approx 2^{k\frac{n}{p}}.
\nonumber
\end{align}
This means $b(f)_{Q_k} = \ell(Q_k)^{\frac{n}{p}}O(f,Q_k) > \lambda_{n,p}\approx 1$, for all $k \in \N$, so 
\[
\lambda_{n,p}^p\sum_{b(f)_Q> \lambda_{n,p}} \ell(Q)^{-n}|Q| = \lambda_{n,p}^p\#\{Q : b(f)_Q>\lambda_{n,p}\} = \infty,
\]
which proves that $\|b(f)\|_{l_n^{p,\infty}(\Rn)} =\infty$.

By modifying the above example to
\[
f=\sum_{k=1}^{\infty}k^{-\frac{\alpha}{p}}2^{k\frac{n}{p}}\chi_{[0,2^{-k}]^n}, \quad 0<\alpha<1,
\]
we obtain a function in $\JNp(\Rn,\cD)$ for which $[b(f)]_{l_n^{p,\infty}(\Rn)}=\infty$.

\begin{example}
\label{Example 5}
\end{example} Here we show  inequality \eqref{thesecond} fails for the case $0<\gamma<n$, $1<p<\infty$. The construction follows the lines of Example~\ref{Example 2}, where we choose a sequence of  finite collections $\cC_k$ of disjoint dyadic cubes in $Q_0 = [0,1]^n$. So let $\left\{n_k\right\}$ be an increasing sequence of natural numbers with
\begin{equation}\label{jnpnumbers}
	q_k:=\lfloor2^{n_k(n-\gamma)}\rfloor, \quad \frac{q_k}{2^{n_k(n-\gamma)}}=1-\varepsilon_k,\quad \varepsilon_k\le2^{-k}, \quad k\in\N, \quad \mbox{and } n_1>\frac{n}{\gamma}.
\end{equation}
The last assumption is meant to guarantee that $2^{n_1(n-\gamma)}<2^{(n_1-1)n}$, which ensures that we can pick $q_1$ different cubes in $\cD_{n_1-1}(Q_0)$ and take a child of each one, giving a collection $\cC_1$ of  $q_1$ cubes in $\cD_{n_1}(Q_0)$, all of which have different parents. We then repeat the process inside each of the selected cubes but this time with $n_2$ and $q_2$, and so on, so at the $k$-th step we have a collection of cubes $\cC_k$ such that
\begin{equation}
\label{lengthofcubesjnp}
	\#\cC_k=q_1\cdots q_k, \quad \ell(Q)=2^{-(n_1+\ldots +n_k)}, \quad Q\in\cC_k,\quad k\in\N.
\end{equation}
Now let 
\[
f=\sum_{k=1}^{\infty}2^{(n_1+ \ldots +n_k)\frac{\gamma}{p}}\sum_{Q\in {\cC_k}}\chi_Q,
\]
and note that 
\[
\int_{Q_0}f=\sum_{k=1}^{\infty}2^{(n_1+ \ldots +n_k)\frac{\gamma}{p}}|\cup\cC_k|=\sum_{k=1}^{\infty}q_1\cdots q_k2^{(\frac{\gamma}{p}-n)(n_1+ \ldots +n_k)}
\le\sum_{k=1}^{\infty}2^{-\frac{\gamma}{p'}(n_1+ \ldots +n_k)}<\infty.
\]

Our next task is to show that $\|b(f)\|_{l_n^{p,\infty}(\Rn)} =\infty$. In order to do so, let $Q\in\cC_k$ and $\Qbar$ be its parent. Note that from the construction the only cube in $\cC_k$ which is contained inside $\Qbar$ is $Q$ itself. Therefore for some constant $c_{\Qbar}$,
\[
f\chi_{\Qbar}=g+c_{\Qbar} \quad g=\sum_{l=k}^{\infty}2^{(n_1+ \ldots +n_l)\frac{\gamma}{p}}\sum_{\substack{Q'\in \cC_l\\Q'\subset Q}}\chi_{Q'},
\]
so as in inequality \eqref{oscillationbelow} of Example~\ref{Example 4} we can replace $f$ by $g$ to get a lower bound on the oscillation:
\begin{align*}
O(f,\Qbar) = O(g,\Qbar) &\ge  
 \frac{|\Qbar \setminus Q|}{2|\Qbar|} \fint_{\Qbar} g \approx \fint_Q g \gtrsim 2^{(n_1+ \ldots +n_k)\frac{\gamma}{p}}.
\end{align*}
This shows, using \eqref{lengthofcubesjnp}, that for some $\lambda \approx 1$, and all $Q \in \cC_k$, $k \in \N$, we have
$$b(f)_{\Qbar} = \ell({\Qbar})^{\frac{\gamma}{p}}O(f,{\Qbar}) > \lambda,$$
so 
\[
\sum_{b_{Q'}(f)>\lambda}\ell(Q')^{n-\gamma}\ge\sum_{k=1}^{\infty}\sum_{Q\in {\cC_k}}\ell(\Qbar)^{n-\gamma}=2^{n-\gamma}\sum_{k=1}^{\infty}2^{-(n-\gamma)(n_1+ \ldots +n_k)}q_1\cdots q_k
=\sum_{k=1}^{\infty}(1-\varepsilon_1)\cdots(1-\varepsilon_k).
\]
Since by \eqref{jnpnumbers} the sum on the right diverges, we have shown $\|b(f)\|_{l_n^{p,\infty}(\Rn)} =\infty$.

Our final goal is to show that $f\in \JNp(\Rn,\cD)$.  To this aim, let $\cP$ be an arbitrary but finite collection of disjoint dyadic cubes.  We can ignore large cubes since only one of those, namely the one containing $Q_0$, will have nonzero oscillation, and that can be estimated as in \eqref{largecubes}.  Thus we may assume $\cP \subset \cD(Q_0)$.  We have to show that
\[
\sum_{Q\in \cP}|Q|O(f,Q)^p\lesssim1,
\]
where the above bound is independent of $\cP$. In order to do this, first we partition $\cP$ into some disjoint subcollections $\cP_k$ as follows. Note that we can remove from $\cP$ the cubes on which $f=0$.  Thus for $Q\in \cP$, if $f$ is not identically $0$, then by the definition of $f$, $Q$ contains a cube in $\cC_k$ for some $k\in \N$.  Let $k_Q$ be the smallest such natural number, and for $k \in \N$ denote $\cP_k = \{Q \in \cP: k_Q = k\}$. 
This gives us a partition of $\cP$ and thus we have
\begin{equation}\label{jnpnormpartion}
	\sum_{Q\in \cP}|Q|O(f,Q)^p=\sum_{k}\sum_{Q\in \cP_k}|Q|O(f,Q)^p.
\end{equation}

Next, for $Q\in \cP_k$ let us estimate the mean oscillation of $f$ on $Q$. First we note that
\[
f\chi_Q=g+\text{const}, \quad g=\sum_{\substack{Q'\subset Q\\Q'\in\cC_k}}\sum_{j=k}^{\infty}2^{(n_1+ \ldots +n_j)\frac{\gamma}{p}}\sum_{\substack{Q''\subset Q'\\Q''\in\cC_j}}\chi_{Q''},
\]
which implies, as in \eqref{oscillationabove}, that
\[
O(f,Q)\lesssim \fint_{Q}{g}\le\frac{1}{|Q|}\sum_{\substack{Q'\subset Q\\Q'\in\cC_k}}\sum_{j=k}^{\infty}2^{(n_1+ \ldots +n_j)\frac{\gamma}{p}} \Big|\bigcup\limits_{\substack{Q''\subset Q'\\Q''\in\cC_j}}Q'' \Big|.
\]
Now for each $Q'\in\cC_k$ occurring in the first sum, we have, from the construction and after applying \eqref{lengthofcubesjnp} and \eqref{jnpnumbers}, the following bound for the second sum:
\begin{align*}
&\sum_{j=k}^{\infty}2^{(n_1+ \ldots +n_j)\frac{\gamma}{p}} \Big|\bigcup\limits_{\substack{Q''\subset Q'\\Q''\in\cC_j}}Q'' \Big|
 = 2^{(\frac{\gamma}{p}-n)(n_1+ \ldots +n_k)}+\sum_{j=k+1}^{\infty}2^{(\frac{\gamma}{p}-n)(n_1+ \ldots +n_j)}q_{k+1}\cdots q_j\\
&\le 2^{(\frac{\gamma}{p}-n)(n_1+ \ldots +n_k)} 
+ 2^{(\frac{\gamma}{p}-n)(n_1+ \ldots +n_k)}\sum_{j=k+1}^{\infty}2^{-\frac{\gamma}{p'}(n_{k+1}+ \ldots +n_j)}(1-2^{-k-1})\cdots (1 -2^{-j})\\
& \lesssim 2^{(\frac{\gamma}{p}-n)(n_1+ \ldots +n_k)}.
\end{align*}
This gives
$$O(f,Q) \lesssim  \frac{\#\{Q' \in \cC_k: Q'\subset Q\}}{|Q|} 2^{(\frac{\gamma}{p}-n)(n_1+ \ldots +n_k)} = 2^{(n_1+ \ldots +n_k)\frac{\gamma}{p}} \fint_Q \sum_{Q'\in\cC_k}\chi_{Q'}.$$
Now, using the fact that the cubes in $\cP_k$ are disjoint, we can bound the inner sum in \eqref{jnpnormpartion} by
\begin{align}
\nonumber
\sum_{Q\in \cP_k}|Q|O(f,Q)^p & \lesssim 2^{(n_1+ \ldots +n_k)\gamma}\sum_{Q\in \cP_k}|Q|\Big(\fint_Q \sum_{Q'\in\cC_k}\chi_{Q'}\Big)^p\\
\nonumber
& \le 2^{(n_1+ \ldots +n_k)\gamma}\sum_{Q\in \cP_k} \int_Q \sum_{Q'\in\cC_k}\chi_{Q'}
= 2^{(n_1+ \ldots +n_k)\gamma} \int_{\cup \cP_k} \sum_{Q'\in\cC_k}\chi_{Q'}\\
&=2^{(n_1+ \ldots +n_k)(\gamma-n)}\#\{Q' \in \cC_k: Q'\subset\cup\cP_k\},
\label{numberboundforel}
\end{align}
where in the last equality \eqref{lengthofcubesjnp} is used. At this point let us introduce the notation
\begin{equation}\label{definitionofa'}
	q'_k=\#\{Q' \in \cC_k: Q'\subset\cup\cP_k\},
\end{equation}
which together with \eqref{jnpnumbers}, \eqref{jnpnormpartion}, \eqref{numberboundforel} and \eqref{definitionofa'} gives
\begin{equation}\label{finalboundforel}
	\sum_{Q\in \cP}|Q|O(f,Q)^p\lesssim \sum_{k}\frac{q'_k}{q_1\cdots q_k},\quad k\in \N.
\end{equation}

We bound the numbers $q'_l$ as follows. First we note that $q'_1\le q_1$. Then since cubes in $\cP_2$ do not contain any cube from $\cC_1$, the total number of cubes from $\cC_2$ that can possibly be found in $\cup\cP_2$ is $(q_1-q'_1)q_2$, and the same reasoning shows that
\[
q'_3\le\left((q_1-q'_1)q_2-q'_2\right)q_3.
\]
Then by induction we get
\begin{equation}
	\frac{q'_k}{q_1\cdots q_k}\le1-\sum_{j=1}^{k-1}\frac{q'_j}{q_1\cdots q_j},\quad k\in \N,
\end{equation}
and this together with \eqref{finalboundforel} shows that
\[
\sum_{Q\in \cP}|Q|O(f,Q)^p\lesssim1.
\]
We can now take the supremum over all $\cP$ to see that  $\|f\|_{\JNp(\Rn,\cD)} \lesssim 1$.

Similarly to the previous example one can check that the function
\[
f=\sum_{k=1}^{\infty}k^{-\frac{\alpha}{p}}2^{(n_1+ \ldots +n_k)\frac{\gamma}{p}}\sum_{Q\in {\cC_k}}\chi_Q,\quad 0<\alpha<1
\]
is in $\JNp(\Rn,\cD)$ but $[b(f)]_{l_\gamma^{p,\infty}(\Rn)}=\infty$.

\begin{claim}
\label{Claim4}
The norm $\|b(f)\|_{l_\gamma^{p,\infty}(\Rn)}$ is not strong enough to imply any higher integrability.  In particular all the inequalities of the form
\begin{equation}
\label{lqembedding}
	\|f\|_{L^{q,\infty}}\le C\|b(f)\|_{l_\gamma^{p,\infty}}, \quad 1<p,q<\infty
\end{equation}
fail to hold.  Moreover, the inequality fails even if the norm on the right-hand side is replaced by $\|b(f)\|_{L^{p,\infty}(\Rnp, \nu_{\gamma})}$.
\end{claim}

\begin{example}
\label{Example 6}
\end{example} Working in dimension $1$, let $\cC \subset \cD(\R)$ be a collection of intervals of bounded sidelength,  and $\{a_I\}_{I\in\cC}$ be a sequence of positive numbers indexed by $\cC$. To this data we associate a function $f$, and a quantity $A$ by writing
\begin{equation}
\label{definitionoff}
	f=\sum_{I\in\cC}a_I\chi_I, \quad A^p=\sup_{\lambda>0}\lambda^p\sum_{\substack{a_I\ell(I)^{\frac{\gamma}{p}}>\lambda\\I\in\cC}}\ell(I)^{1-\gamma}.
\end{equation}

From \eqref{definitionoff} we get the inequality
\[
A\lesssim \|b(f)\|_{l_\gamma^{p,\infty}}
\]
by noting that for $I\in\cC$ with parent $\Ibar$,
\[
O(f,\Ibar)=a_IO(\chi_I,\Ibar)=\frac{1}{2}a_I,
\]
or equivalently
\[
b(f)_{\Ibar}\approx a_I\ell(I)^\frac{\gamma}{p}, \quad I \in \cC.
\]

We will show that if the intervals in $\cC$ are placed far from each other in the sense that for a sufficiently large number $N$, which depends on the data, the $N$-th degree ancestors of intervals in $\cC$ are pairwise disjoint, then we also have the reverse inequality, i.e.\
\begin{equation}
\label{realabstract}
	\|b(f)\|_{l_\gamma^{p,\infty}} \approx A, \quad 1<p<\infty, \quad \gamma\ne0,
\end{equation}
where in the above, the implied constants only depend on $p$ and $\gamma$.

To see this, note that the function $f$ is not constant on a dyadic interval $J$, exactly when it is an ancestor of one the intervals in $\cC$. From the choice of $N$, for $1\le k\le N$ and a fixed interval $I$, its $k$-th degree ancestor, denoted $I^k$,  does not contain any other intervals in $\cD$, and we can use the trivial estimate of the oscillation by the mean to get
\begin{equation}\label{smallparent}
	b(f)_{I^k}\le 2\ell(I^k)^{-1+\frac{\gamma}{p}}\int_{I^k}f\le 2 a_I \ell(I)^{\frac{\gamma}{p}}2^{k(\frac{\gamma}{p}-1)}, \quad 1\le k\le N.
\end{equation}
If $k>N$ we use the same estimate to get
\begin{equation}\label{largeparent}
	b(f)_{I^k}\le 2\ell(I^k)^{-1+\frac{\gamma}{p}}\int f\le2 c(f)d(\cC)2^{k(-1+\frac{\gamma}{p})},
\end{equation}
where $c(f) := \|f\|_{\Lone}$, $d(\cC) := \max\{\ell(I)^{-1+\frac{\gamma}{p}}: I \in \cC\}$. 
Thus for a fixed $\lambda>0$ we have
\[
\lambda^p\sum_{b(f)_J>\lambda}\ell(J)^{1-\gamma}\le  \lambda^p\left(\sum_{1\le k\le N }+\sum_{k>N }\right)\sum_{\substack{b(f)_{I^k}>\lambda\\I\in\cC}}\ell(I^k)^{1-\gamma}=(i)+(ii).
\]
For the first sum, using \eqref{definitionoff} and \eqref{smallparent} gives us
\[
(i)\le \lambda^p\sum_{1\le k\le N }2^{k(1-\gamma)}\sum_{\substack{a_I \ell(I)^{\frac{\gamma}{p}}>2^{k(\frac{-\gamma}{p}+1)-1}\lambda\\I\in\cC}}\ell(I)^{1-\gamma}\lesssim A^p\sum_{1\le k\le N }2^{k(1-p)}\lesssim A^p,
\]
and for the second sum, \eqref{largeparent} implies that
\[
(ii)\le \sum_{k>N}\sum_{\substack{\lambda < 2c(f)d(\cC)2^{k(-1+\frac{\gamma}{p})}\\I\in\cC}} \lambda^p\ell(I^k)^{1-\gamma}
\le2^p c(f)^p d'(\cC) \sum_{k>N} 2^{k(1-p)}\le C(f,\cC,\gamma,p) 2^{N(1-p)}
\]
where $d'(\cC) := d(\cC)^{p}\max\{\ell(I)^{1-\gamma}: I \in \cC\}$. Putting these two bounds together we obtain
\[
\lambda^p\sum_{b(f)_J>\lambda}\ell(J)^{1-\gamma}\lesssim A^p+C(f,\cC,\gamma,p) 2^{N(1-p)}, \quad \lambda>0.
\]
which shows that if $N$ is large enough \eqref{realabstract} holds.  It is important to point out here that the data involved in $C(f,\cC,\gamma,p)$, namely the sizes of the $a_I$ and the lengths of the intervals $I \in \cC$, are independent of the placement of the intervals, determined by the choice of $N$.

Now to disprove \eqref{lqembedding}, fix a natural number $M$ and choose the collection $\cC$ as follows. For each $1\le k\le M$, choose $N_k$ dyadic intervals with length $\ell_k$, where these numbers are determined by
\[
2^{\frac{kp'}{\gamma}}k^{-p'} \le N_k < 2^{\frac{kp'}{\gamma}}k^{-p'} + 1, \quad 2^{-\frac{kp'}{\gamma}} \le \ell_k <2^{-\frac{kp'}{\gamma}+1}.
\]
For each such interval $I$, let 
$$a_I = k^{p'-1},$$
and let $f_M$ be the function defined in \eqref{definitionoff}.  As shown above, given $f_M$, it is possible to assume that all these intervals are placed far from each other so that \eqref{realabstract} holds independently of $M$. 

We will show that given a constant $C$, one can choose $M$ large enough so that \eqref{lqembedding} fails. 
To this aim we bound the quantity $A$ uniformly in $M$. Fixing $\lambda>0$, we write
\[
\sum_{a_I\ell(I)^{\frac{\gamma}{p}}>\lambda}\ell(I)^{1-\gamma}\le \sum_{(2^{-k}k)^{p'-1}>\lambda}(2^{k}k^{-1})^{p'}
\approx(2^{l}l^{-1})^{p'},
\]
where in the above, $l$ is the largest number $k$ in the set over which we are summing. This is a consequence of the geometric growth of the terms in the above sum. From this we can see that 
\[
(2^{-l}l)^{p'-1}\approx\lambda,
\]
and therefore, recalling that $p' - 1 = \frac{p'}{p}$,
\[
\sum_{a_I\ell(I)^{\frac{\gamma}{p}}>\lambda}\ell(I)^{1-\gamma}\lesssim \lambda^{-p},
\]
which together with \eqref{realabstract} proves that
\begin{equation}
\label{bfM}
\|b(f_M)\|_{l_\gamma^{p,\infty}}\lesssim 1
\end{equation}
with constants independent of $M$.

Next, let us estimate the size of the level sets of the function $f_M$. Take $m = \lfloor \frac M 2\rfloor$ and set $\lambda=m^{p'-1}$.  Then by the definitions of $a_I$, $N_k$ and $\ell_k$, we have that
\[
|\left\{f>\lambda\right\}| \ge \sum_{I \in \cC, a_I > \lambda} \ell(I) = \sum_{1 \le k \le M, \; k^{p'-1} >  m^{p'-1}} N_k \ell_k
\geq
\sum_{m<k\le M}k^{-p'}\approx M^{1-p'}=\lambda^{-1}.
\]
Hence for any $q>1$ we have
\begin{equation}
\label{biglqnorm}
	\|f_M\|_{L^{q,\infty}}\gtrsim M^{(q-1)(p'-1)}.
\end{equation}
Taking $M$ sufficiently large, \eqref{bfM} and \eqref{biglqnorm} show that \eqref{lqembedding} fails to hold.

Finally, we show that even if we replace the dyadic norm $\|b(f)\|_{l_\gamma^{p,\infty}(\Rn)}$ by the continuous norm $\|b(f)\|_{L^{p,\infty}(\Rnp, \nu_{\gamma})}$ in \eqref{lqembedding},  the inequality 
still fails. Let us modify the above function as
\[
\ftil=\sum_{I\in\cC} a_I\tilde{\chi}_I,
\]
where in the above $\tilde{\chi}_I$ is a non-negative Lipschitz function supported on $I$ with Lipschitz constant $2\ell(I)^{-1}$, and is equal to 1 on $\frac{1}{2}I$ . Also, this time we place the intervals in $\cC$ in such a way that for large $N$, the dilations $5^NI$ are disjoint for different $I\in \cC$. Then it follows from \eqref{continousdyadic} that it is enough to show that for any dyadic system $\cDtil$ we have 
\begin{equation}
\label{bftil}
\|b(\ftil)\|_{l_\gamma^{p,\infty}(\Rn, \cDtil)}\lesssim A,
\end{equation}
where $A$ is the same as \eqref{definitionoff}. 

The first thing to note is that each interval $I \in \cC$ is covered by at most three intervals $I_1$, $I_2$, and $I_3$ in $\cDtil$ with $\ell(I)/2 < \ell(I_i) \le \ell(I)$. If an interval $J\in \cDtil$ intersects $I$ we must have either $J\subset I_i$ or $J$ is an ancestor of $I_i$ for some $i\in \{1,2,3\}$. Furthermore, for $N \ge 1$, the $N$th ancestor of $I_i$ satisfies $I^N_i\subset 5^NI$, so by the choice of $N$ described above we conclude that for two distinct intervals $I,I'\in\cC$ and their corresponding $I_i,I'_i \in \cDtil$, $i=1,2,3$, the $N$-th ancestors of $I_i$ and $I'_i$. are disjoint. 

Next, fix $\lambda>0$ and note that for intervals of the form $I^k_i$ the estimates \eqref{smallparent} and \eqref{largeparent} remains unchanged, so we only need to estimate the contribution of the intervals $J$ such that for an interval $I\in\cC$, $J\subset I_i$ for some $i\in \{1,2,3\}$. For such an interval, by using the Lipschitz continuity of $\tilde{\chi}_I$, we have
\[
O(\ftil,J)=a_IO(\tilde{\chi}_I,J)\le a_I\frac{2\ell(J)}{\ell(I)},
\]
which implies that for $J\subset I_i$ with $\ell(J)=2^{-k}\ell(I_i)$ we have
\[
b(\ftil)_J = \ell(J)^{\frac{\gamma}{p}}O(\ftil,J) \lesssim 2^{-k(1+\frac{\gamma}{p})}a_I\ell(I)^{\frac{\gamma}{p}}.
\]
To estimate the contribution of the intervals $J \in \cDtil$ of this form to the norm, for $\lambda > 0$ write
\[
\sum_{b(\ftil)_J>\lambda}\ell(J)^{1-\gamma}
\le \sum_{i=1}^{3}\sum_{k\ge0}\sum_{\substack{a_I\ell(I)^{\frac{\gamma}{p}}\gtrsim2^{k(1+\frac{\gamma}{p})}\lambda\\\ell(J)=2^{-k}\ell(I_i)\\I\in\cC}}\ell(J)^{1-\gamma}\le 3\sum_{k\ge0}2^{k\gamma}\sum_{\substack{a_I\ell(I)^{\frac{\gamma}{p}}\gtrsim2^{k(1+\frac{\gamma}{p})}\lambda\\I\in\cC}}\ell(I)^{1-\gamma},
\]
which after using the definition of $A$ in \eqref{definitionoff} gives us
\[
\lambda^p\sum_{b(\ftil)_J>\lambda}\ell(J)^{1-\gamma}\lesssim \lambda^p A^p \sum_{k\ge0}2^{k\gamma}(2^{k(1+\frac{\gamma}{p})}\lambda)^{-p}
=
A^p\sum_{k\ge0}2^{-kp}\lesssim A^p,
\]
proving \eqref{bftil}. Finally, since $\tilde{\chi}_I$ is equal to $1$ on half of $I\in\cC$, we conclude that \eqref{biglqnorm} holds true for $\ftil$ as well.  As before, taking $M$ sufficiently large, this shows that \eqref{lqembedding} fails for the continuous norm as well.

\bibliographystyle{amsplain}

\end{document}